\documentclass[11pt]{amsart}
\pdfoutput=1

\usepackage{amssymb}
\usepackage{amsthm}
\usepackage{amsfonts}
\usepackage{amsmath}
\usepackage{pmboxdraw}
\usepackage{verbatim} 
\usepackage{graphicx}
\usepackage{color}
\usepackage[colorlinks=true, citecolor=blue, filecolor=black, linkcolor=black, urlcolor=black]{hyperref}
\usepackage{cite}
\usepackage[normalem]{ulem}
\usepackage{subcaption}
\usepackage{bm}
\usepackage{mathtools}
\mathtoolsset{showonlyrefs}
\usepackage{todonotes}
\usepackage{bbm}

\newcommand{\RN}[1]{%
	\textup{\uppercase\expandafter{\romannumeral#1}}%
}

\def\bp{{\bar\partial}}

\def\pa{\partial}

\def\wt{\widetilde}

\def\Re{ \mathrm{Re}}

\def\OO{\mathcal{O}}

\def\C{\mathbb{C}}

\def\R{\mathbb{R}}

\theoremstyle{plain}
\newtheorem*{thm*}{Theorem}
\newtheorem{thm}{Theorem}[section]
\newtheorem{lem}[thm]{Lemma}
\newtheorem{cor}[thm]{Corollary}
\newtheorem{prop}[thm]{Proposition}
\newtheorem*{prop*}{Proposition}
\newtheorem*{lem*}{Lemma}

\theoremstyle{definition}
\newtheorem*{eg*}{Example}
\newtheorem*{egs*}{Examples}
\newtheorem*{def*}{Definition}

\theoremstyle{remark}
\newtheorem*{rmk*}{Remark}
\newtheorem*{rmks*}{Remarks}

\makeatletter
\@namedef{subjclassname@2020}{\textup{2020} Mathematics Subject Classification}
\makeatother

\newcommand{\bfR}{\mathbf{R}}

\newcommand{\calK}{{\mathcal K}}

\newcommand{\calN}{{\mathcal N}}

\newcommand{\bfkappa}{{\bm \varkappa}}

\newcommand{\erfc}{\operatorname{erfc}}
\newcommand{\erf}{\operatorname{erf}}

\newcommand{\N}{\mathbb{N}}
\newcommand{\Z}{{\mathbb Z}}

\newcommand{\eps}{{\varepsilon}}

\newcommand{\re}{\operatorname{Re}}
\newcommand{\im}{\operatorname{Im}}

\newcommand{\Prob}{{\mathbf{P}}}

\newcommand{\Tr}{\operatorname{Tr}}

\newcommand{\diam}{\operatorname{diam}}

\newcommand{\Pf}{{\textup{Pf}}}
\newcommand{\curvature}{\varkappa}

\newcommand{\abs}[1]{\lvert#1\rvert}
\numberwithin{equation}{section}

\begin{document}
\title[Symplectic elliptic Ginibre ensemble]{Universal scaling limits of the symplectic elliptic Ginibre ensemble 
}


\author{Sung-Soo Byun}
\address{Center for Mathematical Challenges, Korea Institute for Advanced Study, 85 Hoegiro, Dongdaemun-gu, Seoul 02455, Republic of Korea}
\email{sungsoobyun@kias.re.kr}

\author{Markus Ebke}
\address{Department of Mathematics, Friedrich-Alexander-Universität Erlangen-Nürnberg, Cauerstrasse 11, 91058 Erlangen, Germany}
\email{markus.ebke@fau.de}

\thanks{The authors are grateful to the DFG-NRF International Research Training Group IRTG 2235 supporting the Bielefeld-Seoul graduate exchange programme. Furthermore, Sung-Soo Byun was partially supported by Samsung Science and Technology Foundation (SSTF-BA1401-51), by the National Research Foundation of Korea (NRF-2019R1A5A1028324) and by a KIAS Individual Grant (SP083201) via the Center for Mathematical Challenges at Korea Institute for Advanced Study.}
\begin{abstract}
We consider the eigenvalues of symplectic elliptic Ginibre matrices which are known to form a Pfaffian point process whose correlation kernel can be expressed in terms of the skew-orthogonal Hermite polynomials. 
We derive the scaling limits and the convergence rates of the correlation functions at the real bulk/edge of the spectrum, which in particular establishes the local universality at strong non-Hermiticity.
Furthermore, we obtain the subleading corrections of the edge correlation kernels, which depend on the non-Hermiticity parameter contrary to the universal leading term. 
Our proofs are based on the asymptotic behaviour of the complex elliptic Ginibre ensemble due to Lee and Riser as well as on a version of the Christoffel-Darboux identity, a differential equation satisfied by the skew-orthogonal polynomial kernel. 
\end{abstract}

\maketitle


\section{Introduction and main results}

What has been understood about non-Hermitian random matrices with symplectic symmetry? 
To name a few, we should start with the celebrated work of Ginibre \cite{ginibre1965statistical}. 
There, he introduced Gaussian random matrices with quaternion entries (now known as the symplectic Ginibre ensemble) and derived the joint probability distribution function $\Prob_N$ for their eigenvalues $\boldsymbol{\zeta}=(\zeta_1,\cdots,\zeta_N) \in \C^N$.
It is given by
\begin{align} \label{Gibbs}
\begin{split}
d\Prob _N(\boldsymbol{\zeta}) &= \frac{1}{Z_N} \prod_{j>k=1}^{N} \abs{\zeta_j-\zeta_k}^2 \abs{\zeta_j-\overline{\zeta}_k}^2 \prod_{j=1}^{N} \abs{\zeta_j-\overline{\zeta}_j}^2 e^{ -N  Q(\zeta_j) } \,  dA(\zeta_j), 
\end{split}
\end{align}
where $Z_N$ is the partition function and $dA(\zeta):=\tfrac{1}{\pi}d^2\zeta$.
Here $Q(\zeta)=|\zeta|^2$, but one may consider an arbitrary potential $Q: \C \to \R$ as long as it satisfies suitable potential theoretic assumptions. 
From the statistical physics point of view, this eigenvalue distribution \eqref{Gibbs} can be interpreted as a two-dimensional Coulomb gas with an additional complex conjugation symmetry, see e.g.\ \cite[Section 15.9]{forrester2010log} and \cite{MR3458536,kiessling1999note}.

Regarding the macroscopic (global) properties of the symplectic Ginibre ensemble, it is well known that as $N \to \infty$, the eigenvalues tend to be uniformly distributed on the centered disc with radius $\sqrt{2}$, known as the circular law. 
Such a statement was extended to the ensemble \eqref{Gibbs} with general $Q$ in \cite{MR2934715}, where it was shown that as $N \to \infty$, the system tends to minimise the weighted logarithmic energy functional \cite{ST97}, the continuum limit of its discrete Hamiltonian. 
In particular, in the large system, the eigenvalues condense in a compact set $S \subset \C$ called the droplet. 
Moreover, due to a well-known fact from the logarithmic potential theory \cite{ST97}, it follows that the density inside the droplet with respect to the area measure $dA$ is given by $\Delta Q/2$, where we use the convention $\Delta=\pa \bp$. 

Turning to the microscopic (local) properties, contrary to the complex Ginibre ensemble whose bulk scaling limit was obtained already in Ginibre's original work \cite{ginibre1965statistical}, the equivalent level of analysis for the symplectic Ginibre ensemble was obtained long after. 
To our knowledge, it first appeared in the second edition of Mehta's book \cite{Mehta}, where the bulk scaling limit was derived at the origin of the spectrum.
More precisely, it is known that the particle system \eqref{Gibbs} forms a Pfaffian point process whose correlation functions can be expressed in terms of a certain $2\times2$ matrix-valued kernel (see \eqref{bfR Pfa} below), and the large-$N$ limit of the kernel was computed in \cite{Mehta}. (See also \cite{MR3066113} for the scaling limits of the products of Ginibre matrices at the singular point, the origin.)

Beyond the radially symmetric potentials, an elliptic generalisation of the symplectic Ginibre ensemble was studied by Kanzieper \cite{MR1928853}.  
Such an extension is called the \emph{elliptic Ginibre ensemble} and is associated with the potential
\begin{equation} \label{Q elliptic}
Q(\zeta):=\tfrac{1}{1-\tau^2} ( |\zeta|^2-\tau \re \zeta^2 ), \qquad \tau \in [0,1). 
\end{equation}
(cf. see \cite{katori2019two} for a different type of elliptic extension of planar point processes.)
The associated droplet $S$ is given by the elliptic disc
\begin{equation} \label{droplet}
S:= \{ x+iy \in \C \, | \, ( \tfrac{x}{\sqrt{2}(1+\tau)} )^2+( \tfrac{y}{\sqrt{2}(1-\tau)} )^2 \le 1 \},
\end{equation}
which is also known as the elliptic law, see the top part of
Figure~\ref{Fig. QEG}. 
In \cite{MR1928853}, Kanzieper introduced the formalism of skew-orthogonal polynomials to construct the correlation kernel and showed that the elliptic Ginibre ensemble can be exactly solved using the Hermite polynomials. 
Furthermore, based on this formalism, he obtained the scaling limits at the origin in the maximally non-Hermitian regime where $\tau=0$ as well as in the almost-Hermitian regime where $1-\tau=\OO(N^{-1})$. 
The analogous result for any fixed $\tau \in [0,1)$ was obtained in \cite{akemann2021skew}. We also refer to \cite{MR2180006} for the chiral counterpart of such a result, which can be solved using the generalised Laguerre polynomials.

\begin{figure}[h!]
		\begin{center}	
			\includegraphics[width=0.6\textwidth]{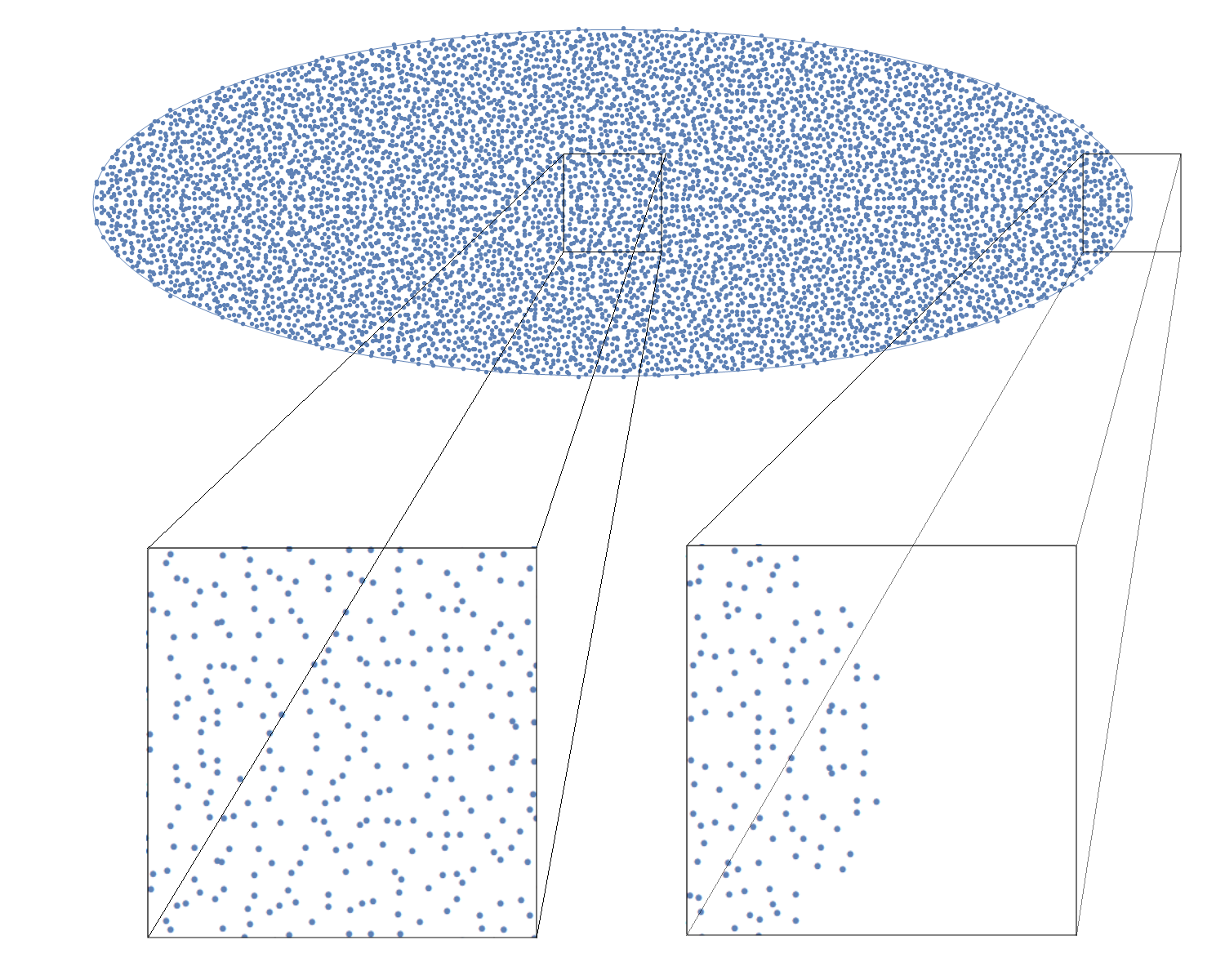}
		\end{center}
	\caption{A single sample of eigenvalues of the elliptic Ginibre ensemble. Here, $N=4096$ and $\tau=1/2$.} \label{Fig. QEG} 
\end{figure}

While the above-mentioned works mainly focused on various scaling limits at the origin, recently, there have been several studies on the scaling limits of the symplectic ensembles away from the origin as well. 
One remarkable result in this direction was due to Akemann, Kieburg, Mielke, and Prosen \cite{akemann2019universal}, where the authors showed that away from the real axis, the eigenvalue statistics of the symplectic and complex Ginibre ensembles are equivalent in the large-$N$ limit, see also \cite{MR1986426} for a similar statement on the fluctuations of the maximal modulus.
Moreover very recently, the edge scaling limit of the symplectic Ginibre ensemble at the right/left endpoint of the spectrum was obtained independently in \cite{akemann2021scaling} and \cite{BL2,Lysychkin}. 
These works can be thought of as the maximally non-Hermitian analogue of the previous work \cite{MR3192169}, which investigated the edge scaling limits of the elliptic Ginibre ensemble in the almost-Hermitian regime where $1-\tau=\OO(N^{-\frac13})$.
However at strong non-Hermiticity with general (fixed) $\tau \in [0,1)$, the scaling limits along the real axis as well as their convergence rates have not been discovered yet and we aim to contribute to these problems.  
(After this work, the derivation of scaling limits at weak non-Hermiticity was extended to the entire bulk and edge along the real axis in \cite{byun2021wronskian}.)

\medskip 

Let us now be more precise in introducing our main results.
We consider the matrix model
\begin{equation} \label{Matrix Model}
    X = \sqrt{1 + \tau} \, S + \sqrt{1 - \tau} \, T, \qquad
    S = \frac{1}{2}\Big( G_1 + G_1^{\dagger} \Big),\ T = \frac{1}{2}\Big( G_2 - G_2^{\dagger} \Big)
\end{equation}
where $G^{\dagger}$ denotes the conjugate transpose and $G_1$ and $G_2$ are independent random matrices from the symplectic Ginibre ensemble (see \cite{ginibre1965statistical} and \cite[Section~15.2]{Mehta}), i.e.\ they have the form
\begin{equation*}
    G = \begin{pmatrix}
        A & B \\
        -\overline{B} & \overline{A}
    \end{pmatrix} \in \C^{2 N \times 2 N}, \qquad
    A_{ij}, B_{ij} \sim \calN\Big(0, \frac{1}{2 N}\Big) + i \calN\Big(0, \frac{1}{2 N}\Big)
\end{equation*}
Note that due to this form our matrix $X$ has $2N$ eigenvalues that come in complex conjugate pairs. 
However, only one value from each pair is needed to characterise the spectrum of $X$.
We used a single realisation of such a matrix $X$ to create Figure~\ref{Fig. QEG} where we plotted all $2N$ eigenvalues.

The matrix probability density is given by
\begin{equation*}
    d\Prob(X) = \exp\Big( -\frac{N}{1-\tau^2} \big[ \Tr(X X^{\dagger}) - \tau \re \Tr(X^2) \big] \Big) \, dX.
\end{equation*}
Then diagonalizing $X$ as in Ginibre's paper \cite{ginibre1965statistical} or writing it in Schur form (compare \cite[Appendix A.33]{Mehta}) we can derive the eigenvalue distribution \eqref{Gibbs} with \eqref{Q elliptic}.
We remark that for general potentials $Q$ it is not clear how to construct a matrix model that leads to the eigenvalue distribution \eqref{Gibbs}.
Apart from the symplectic elliptic Ginibre ensemble and its chiral version \cite{MR2180006} other known cases include the truncated unitary symplectic ensemble \cite{BL} and products of symplectic Ginibre matrices \cite{MR3066113}.

In the sequel, we shall focus on the ensemble \eqref{Gibbs} with the elliptic potential \eqref{Q elliptic}. 
It is well known (see e.g.\ \cite{MR1928853}) that the $k$-point correlation function
\begin{equation}
\bfR_{N,k}(\zeta_1,\cdots, \zeta_k) := \frac{N!}{(N-k)!} \int_{\C^{N-k}} \Prob _N(\zeta_1,\dots,\zeta_N) \prod_{j=k}^N \, dA(\zeta_j) 
\end{equation}
has the structure
\begin{equation} \label{bfR Pfa}
\bfR_{N,k}(\zeta_1,\cdots, \zeta_k) =\prod_{j=1}^{k} (\overline{\zeta}_j-\zeta_j)  \Pf \Big[ 
e^{ -NQ(\zeta_j)/2-NQ(\zeta_l)/2 } 
\begin{pmatrix} 
\bfkappa_N(\zeta_j,\zeta_l) & \bfkappa_N(\zeta_j,\bar{\zeta}_l)
\smallskip 
\\
\bfkappa_N(\bar{\zeta}_j,\zeta_l) & \bfkappa_N(\bar{\zeta}_j,\bar{\zeta}_l) 
\end{pmatrix}  \Big]_{ j,l=1,\cdots, k }.
\end{equation}
Here the two-variable function $\bfkappa_N$ is called the \emph{(skew) pre-kernel}. 

To describe the local statistics at a given point $p \in \C$, it is convenient to define the rescaled point process $\boldsymbol{z}=(z_1,\dots,z_N)$ as
\begin{equation} \label{rescaling}
z_j:=e^{-i \theta}\sqrt{N \delta} \cdot (\zeta_j-p), \qquad 
\delta:=\dfrac{\Delta Q(p)}{2}=\frac{1}{2(1-\tau^2)},
\end{equation}
where $\theta\in\R$ is the angle of the outward normal direction at the boundary if $p\in\partial S$ (and otherwise $\theta=0$). 
Here, the rescaling factor $\sqrt{N \delta}$ is chosen according to the macroscopic density of the ensemble at the point $p$. 
See Figure~\ref{Fig. QEG} for an illustration of the rescaled process.
The correlation function $R_{N,k}$ of the process $\boldsymbol{z}$ is given by 
\begin{align}
\begin{split}
R_{N,k}(z_1,\cdots, z_k)&:=\frac{1}{(N\delta)^{k}  } \bfR_{N,k}(\zeta_1,\cdots,\zeta_k).
\end{split}
\end{align}
(In the end, this rescaling makes $R_{N,1}$ close to $1$ in the complex bulk.)
In particular, if $p \in \R$, it can be written as 
\begin{align}
\begin{split}
R_{N,k}(z_1,\cdots, z_k) &=\prod_{j=1}^{k} (\bar{z}_j-z_j)  \Pf \Big[ e^{ -\frac{N}{2} ( Q( p+\frac{z_j}{\sqrt{N \delta}} )+Q( p+ \frac{z_l}{\sqrt{N \delta}} ) ) } \begin{pmatrix} 
\kappa_N(z_j,z_l) & \kappa_N(z_j,\bar{z}_l) \\
\kappa_N(\bar{z}_j,z_l) & \kappa_N(\bar{z}_j,\bar{z}_l) 
\end{pmatrix}  \Big]_{ j,l=1,\cdots k },
\end{split}
\end{align}
where the rescaled pre-kernel $\kappa_N$ reads 
\begin{align} \label{kappa rescaling}
\begin{split}
\kappa_N(z,w) :=\frac{1}{ (N \delta)^{ \frac32 } } \bfkappa_N(\zeta,\eta).
\end{split}
\end{align}
In particular, the one-point function $R_N \equiv R_{N,1}$ for the rescaled point process has the form 
\begin{equation}
R_N(z) =(\bar{z}-z) e^{-N Q( p+\tfrac{z}{\sqrt{N \delta}} ) } \, \kappa_N(z,\bar{z}).
\end{equation}

Recall that the $k$-th Hermite polynomial $H_k$ is given by
\begin{equation}
H_k(z):=(-1)^k e^{z^2} \frac{d^k}{dz^k} e^{-z^2}=k! \sum_{m=0}^{ \lfloor k/2 \rfloor } \frac{(-1)^m}{ m! (k-2m)! } (2z)^{k-2m}. 
\end{equation}
Then it follows from \cite{MR1928853} that the pre-kernel $\kappa_N$ has the canonical representation
\begin{align}
\begin{split}
 \label{kappaN cano}
\kappa_N(z,w)&=  \sqrt{2}(1+\tau)  \sum_{k=0}^{N-1}  \frac{ (\tau/2)^{k+\frac12} }{(2k+1)!!}   H_{2k+1} \Big(  \sqrt{ \tfrac{N}{2\tau} } p+\sqrt{\tfrac{1-\tau^2}{\tau}}z \Big)  
 \sum_{l=0}^k  \frac{(\tau/2)^l}{(2l)!!} H_{2l}  \Big(  \sqrt{ \tfrac{N}{2\tau} } p+\sqrt{\tfrac{1-\tau^2}{\tau}}w  \Big)  
\\
&\quad - \sqrt{2}(1+\tau)  \sum_{k=0}^{N-1}  \frac{ (\tau/2)^{k+\frac12} }{(2k+1)!!}   H_{2k+1} \Big(  \sqrt{ \tfrac{N}{2\tau} } p+\sqrt{\tfrac{1-\tau^2}{\tau}}w \Big)  
 \sum_{l=0}^k  \frac{(\tau/2)^l}{(2l)!!} H_{2l}  \Big(  \sqrt{ \tfrac{N}{2\tau} } p+\sqrt{\tfrac{1-\tau^2}{\tau}}z  \Big).  
\end{split}
\end{align}
See Subsection~\ref{Subsec_SOP} for further details.

In our first result Proposition~\ref{Prop_ODE tau}, we obtain a differential equation satisfied by the pre-kernel $\kappa_N,$ which can be recognised as a version of the Christoffel-Darboux formula, cf. \cite{MR3450566,akemann2021scaling,MR1917675,byun2021lemniscate}. 

\begin{prop} \label{Prop_ODE tau}
For each $N$ and for any $p \in \C$, the (canonical) pre-kernel $\kappa_N$ satisfies the differential equation
\begin{align} \label{ODE tau}
\begin{split}
\partial_z \kappa_N(z,w)&= \Big( \sqrt{ \tfrac{2(1-\tau) N}{1+\tau} } \, p+2(1-\tau)z\Big)  \kappa_N(z,w)
\\
&\quad +  2\sqrt{1-\tau^2} \, \sum_{k=0}^{2N-1}  \frac{ (\tau/2)^{k} }{k!}   H_{k}\Big( \sqrt{ \tfrac{N}{2\tau} } p+\sqrt{\tfrac{1-\tau^2}{\tau}}z \Big) H_{k}  \Big( \sqrt{ \tfrac{N}{2\tau} } p+\sqrt{\tfrac{1-\tau^2}{\tau}}w \Big) 
\\
&\quad -  2\sqrt{1-\tau^2}   \frac{ (\tau/2)^{N} }{(2N-1)!!}  H_{2N}  \Big( \sqrt{ \tfrac{N}{2\tau} } p+\sqrt{\tfrac{1-\tau^2}{\tau}}z \Big)  \sum_{l=0}^{N-1}  \frac{(\tau/2)^l}{(2l)!!} H_{2l}  \Big( \sqrt{ \tfrac{N}{2\tau} } p+\sqrt{\tfrac{1-\tau^2}{\tau}}w \Big) .
\end{split}
\end{align}
\end{prop}

\begin{rmk*}
   It follows from 
	\begin{equation*}
	( \tfrac{\tau}{2} )^{k/2}  H_{k}\Big( \sqrt{ \tfrac{N}{2\tau} } p+\sqrt{\tfrac{1-\tau^2}{\tau}}z \Big) \sim ( \sqrt{N}p + \sqrt{2}z )^{k} ,\qquad (\tau \to 0)
	\end{equation*}
	that for $\tau=0$, the equation \eqref{ODE tau} reduces to 
    \begin{align} \label{ODE tau0}
	\begin{split}
	\partial_z \kappa_N(z,w)&= (\sqrt{2N}p + 2z) \kappa_N(z,w)
	\\
	&\quad + 2 \sum_{k=0}^{2N-1} \frac{(\sqrt{N}p + \sqrt{2}z)^{k} (\sqrt{N}p + \sqrt{2}w)^{k}}{k!}-2 \frac{(\sqrt{N}p + \sqrt{2}z)^{2N}}{(2N-1)!!} \sum_{l=0}^{N-1} \frac{(\sqrt{N}p + \sqrt{2}w)^{2l}}{(2l)!!}. 
	\end{split}
	\end{align}
	This differential equation was utilized in \cite{akemann2021scaling}. See also \cite{MR1762659} for a related statement in the other extremal case when $\tau=1$. 
    
    It is worth pointing out that the inhomogeneous term in the second line of \eqref{ODE tau} corresponds to the (holomorphic) kernel of the complex elliptic Ginibre ensemble with dimension $2N$. Such a relation has been observed in other models as well, which include the Laguerre \cite{MR2180006,osborn2004universal,MR4229527} and the Mittag-Leffler ensembles \cite{akemann2021scaling,ameur2018random,chau1998structure}.   
    
	Let us remark that the counterpart of such a differential equation for the kernel of the complex elliptic Ginibre ensemble was obtained in \cite[Proposition 2.3]{MR3450566}. (See also \cite[Subsection 3.1]{byun2021lemniscate} for an alternative derivation in a more general framework.)
	As it allows to perform a suitable asymptotic analysis, this formula has turned out to be very useful in various situations, see e.g.\ \cite{AB21,byun2021real,riser2013universality}. 
	In a similar spirit, we emphasise that Proposition~\ref{Prop_ODE tau} can also be applied to other cases including the almost-Hermitian regime, see \cite{byun2021wronskian}. 
	(We also mention that a different approach for the complex elliptic Ginibre ensemble was recently developed in \cite{akemann2022elliptic}.)
\end{rmk*}

In our main result Theorem~\ref{Thm_local ell} below, we establish the large-$N$ asymptotic of the kernel. This is based on the asymptotic analysis of the differential equation \eqref{ODE tau}, see Corollary~\ref{Cor_ODE transform} and Proposition~\ref{Prop_rN}. 
Recall that by \eqref{droplet}, we have $S \cap \R=[-\sqrt{2}(1+\tau),\sqrt{2}(1+\tau)].$

\begin{thm} \label{Thm_local ell}
Let $\tau \in [0, 1)$ and $p \in \R$. 
Then there exists a pre-kernel $\wt{\kappa}_N$ such that 
\begin{equation}
R_{N,k}(z_1,\cdots, z_k) =\prod_{j=1}^{k} (\bar{z}_j-z_j) \, \Pf \Big[ e^{-|z|^2-|w|^2} \begin{pmatrix} 
\wt{\kappa}_N(z_j,z_l) & \wt{\kappa}_N(z_j,\bar{z}_l) 
\smallskip 
\\
\wt{\kappa}_N(\bar{z}_j,z_l) & \wt{\kappa}_N(\bar{z}_j,\bar{z}_l) 
\end{pmatrix}  \Big]_{ j,l=1,\cdots k }
\end{equation}
and that satisfies the following asymptotic behaviours as $N \to \infty$.
\begin{itemize}
    \item \textbf{\textup{(The real bulk)}} For $|p|<\sqrt{2}(1+\tau)$, there exists an $\eps>0$ such that
    \begin{equation}
    \wt{\kappa}_N(z,w) = \kappa^{\R}_{\textup{bulk}}(z,w) + \OO(e^{-N\eps}),
    \end{equation}
    where 
    \begin{equation} \label{kappa bulk}
    \kappa^{\R}_{\textup{bulk}}(z,w) := \sqrt{\pi} e^{z^2+w^2} \erf(z-w).
    \end{equation}
    \item \textbf{\textup{(The real edge)}} For $p = \pm \sqrt{2}(1+\tau)$ and an arbitrary $\eps > 0$ it holds
    \begin{equation}
    \wt{\kappa}_N(z,w) = \kappa^{\R}_{\textup{edge}}(z,w) + \tfrac{1}{\sqrt{N}} \kappa^{\R,1/2}_{\textup{edge}}(z,w) + \OO(N^{-1+\eps}),
    \end{equation}
    where 
    \begin{equation} \label{kappa edge}
    \kappa^{\R}_{\textup{edge}}(z,w) := e^{2zw} \int_{-\infty}^{0} e^{-s^2}\sinh(2s(w-z))\erfc(z+w-s) \, ds
    \end{equation}
    and 
    \begin{align} \label{kappa edge sub}
    \begin{split}
    \kappa^{\R,1/2}_{\textup{edge}}(z,w) &:= \tfrac{1}{12 \sqrt{2}} (\tfrac{1+\tau}{1-\tau})^{3/2}\, e^{z^2+w^2}  \\
    &\quad \times \Big[ \Big( (2z^2+\tfrac{1-2\tau}{1+\tau})e^{-2z^2}\erfc(\sqrt{2}w) + 2\sqrt{\tfrac{2}{\pi}}w \, e^{-2(z^2+w^2)} \Big)-\Big(z \leftrightarrow w\Big)\Big].
    \end{split}
    \end{align}
    \smallskip 
    \item \textbf{\textup{(The real axis outside the droplet)}}  For $|p|>\sqrt{2}(1+\tau)$, there exists an $\eps>0$ such that
    \begin{equation}
    \wt{\kappa}_N(z,w) = \OO(e^{-N\eps}).
    \end{equation}
\end{itemize}
In all three cases the error terms are uniform for $z, w$ in compact subsets of $\C$.
\end{thm}

Since the leading terms of the pre-kernels are independent of $\tau$ and $p$ (for the real bulk case), Theorem~\ref{Thm_local ell} affirms the local bulk/edge universality along the real axis.
For the complex elliptic Ginibre ensemble, the analogous result on the fine asymptotic behaviours of the kernels were obtained by Lee and Riser \cite{MR3450566}. We also refer to \cite{MR2208159} for similar results on Gaussian symplectic ensemble when $\tau=1$.

While we have focused on the case $p \in \R$ in Theorem~\ref{Thm_local ell}, it would be also interesting to investigate scaling limits away from the real axis, i.e.\ $p \in \C \setminus \R.$ 
Intuitively, one can expect that for $p \in \C \setminus \R$, the local statistics of the symplectic and complex elliptic Ginibre ensemble are equivalent in the large-$N$ limit, see \cite{akemann2019universal,akemann2021scaling} for the discussion on the Ginibre ensembles $(\tau=0)$.  

\begin{figure}[hp!]
	\begin{subfigure}[h]{0.6\textwidth}
		\begin{center}
			\includegraphics[width=0.8\textwidth]{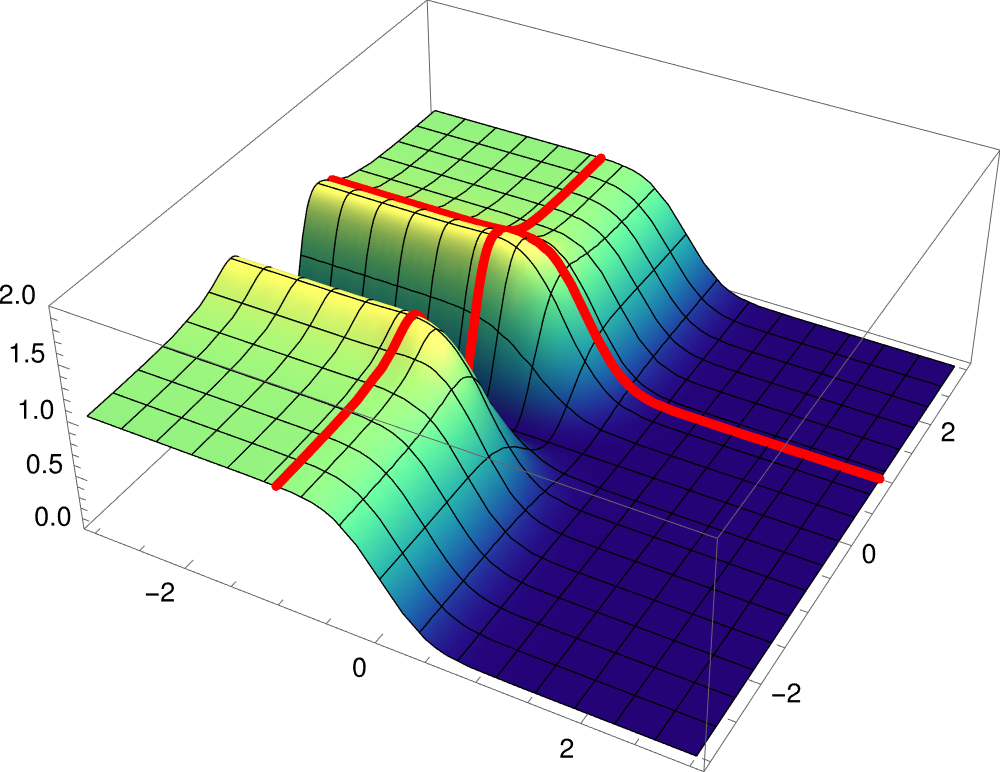}
		\end{center} \subcaption{$R(z)$ for $p$ at the edge}
	\end{subfigure}
	
	\vspace{1em}
	
	\begin{subfigure}{0.48\textwidth}
		\begin{center}
			\includegraphics[width=0.8\textwidth]{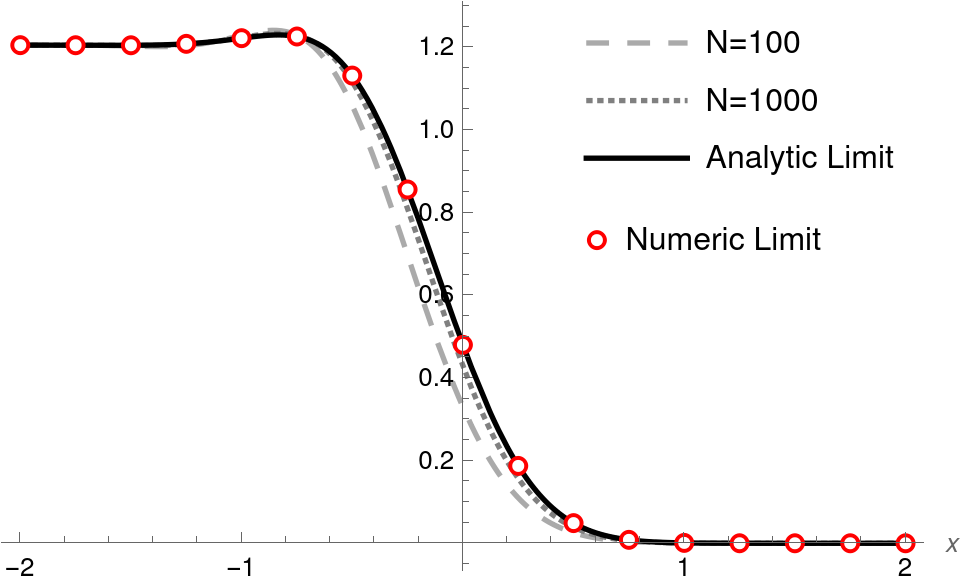}
		\end{center}
		\subcaption{$R_N(z)$ when $\im z=1$}
	\end{subfigure}	
	\begin{subfigure}[h]{0.48\textwidth}
		\begin{center}
			\includegraphics[width=0.8\textwidth]{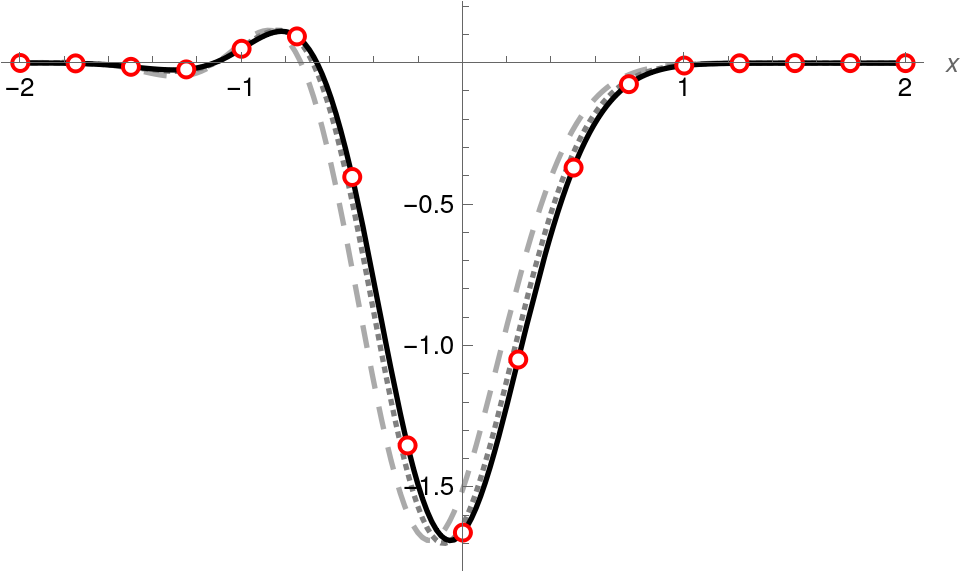}
		\end{center}
		\subcaption{$\sqrt{N}(R_N(z)-R(z))$ when $\im z=1$}
	\end{subfigure}
	
	\begin{subfigure}{0.48\textwidth}
	\begin{center}	
			\includegraphics[width=0.8\textwidth]{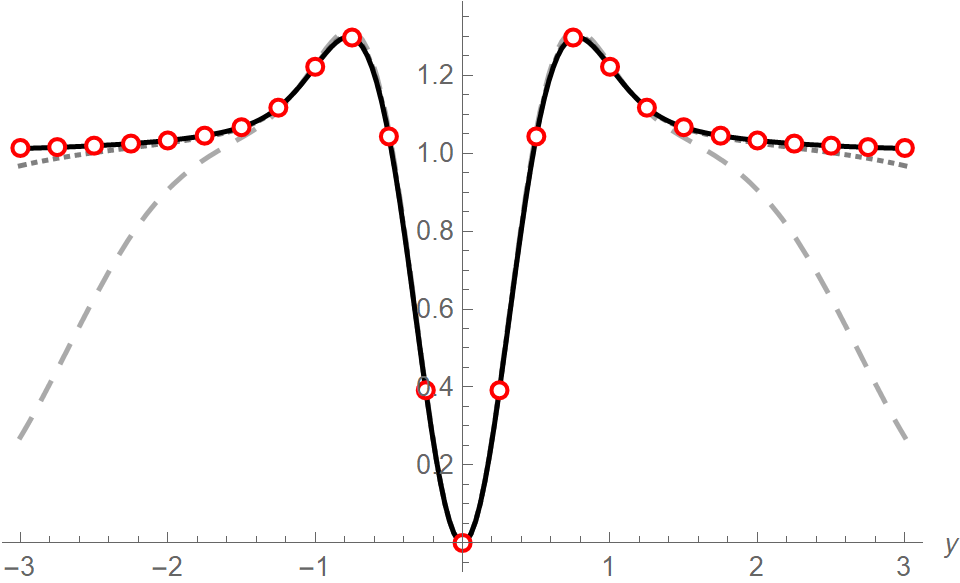}
		\end{center}
		\subcaption{$R_N(z)$ when $\re z=-1$}
	\end{subfigure}	
	\begin{subfigure}[h]{0.48\textwidth}
		\begin{center}
			\includegraphics[width=0.8\textwidth]{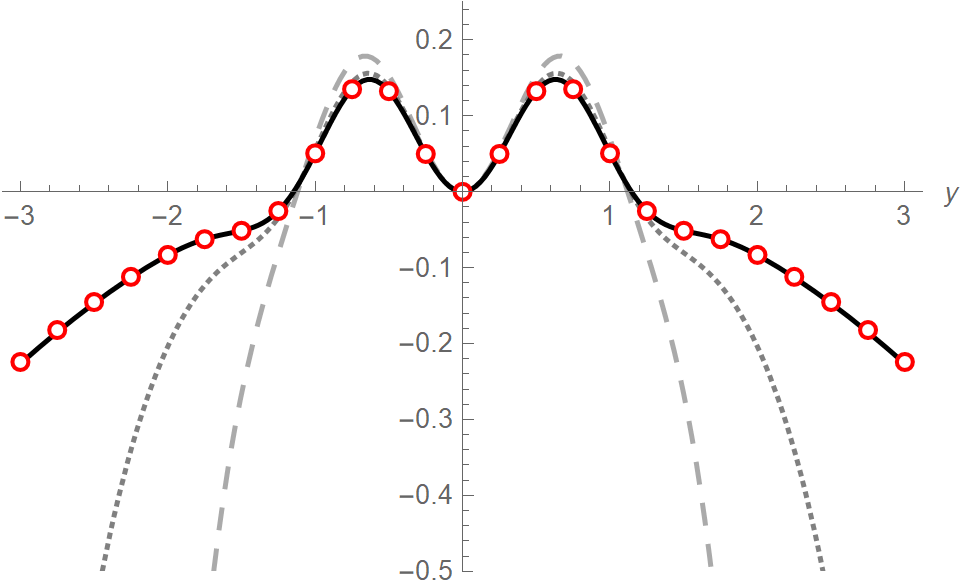}
		\end{center}
		\subcaption{$\sqrt{N}(R_N(z)-R(z))$ when $\re z=-1$}
	\end{subfigure}
	\caption{ The plot (A) shows the graph of the limiting microscopic density $R$ in \eqref{R edge}. 
	The plot (B) is the graph of $R_N$ and its comparison with $R$. The plot (C) shows the same graph for $\sqrt{N}(R_N-R)$ and $R^{(1/2)}$ in \eqref{R12 edge}. In both cases $\im z=1$.
	The method to obtain the numeric limit (red circles) is described in the main text.
	The plots (D),(E) are analogous figures along the cross-section $\re z=1$.
	For all plots we chose $\tau=1/3$. 
	}
	\label{Fig. Density convergence at edge}
\end{figure}

\begin{rmk*}
As noted in \cite{akemann2021scaling}, the limiting pre-kernels in Theorem~\ref{Thm_local ell} can be written in the following unified way
\begin{equation}  \label{kappa lim}
\kappa_a^\R(z,w) :=\frac{ e^{z^2+w^2}   }{\sqrt{2}} \int_{-\infty}^a  e^{-2(z-u)^2  } \erfc ( \sqrt{2}(w-u) )-(z \leftrightarrow w) \, du,	\qquad a= 
\begin{cases}
\infty   &\textup{if} \quad |p| < \sqrt{2}(1+\tau),
\smallskip
\\
0 &\textup{if} \quad |p| = \sqrt{2}(1+\tau),
\smallskip 
\\
-\infty &\textup{if} \quad |p| > \sqrt{2}(1+\tau).
\end{cases}
\end{equation}
For an arbitrary $a \in \R$, the pre-kernel $\kappa_a^\R$ can be obtained by rescaling the process at the $N$-dependent point 
\begin{equation}
p=\pm \sqrt{2}(1+\tau) \mp a \sqrt{ \tfrac{2(1-\tau^2)}{N} }.
\end{equation}
(See also \cite{akemann2021scaling,MR4030288} for further motivations on zooming the point process at a moving location.)
\end{rmk*}

As an immediate consequence of Theorem~\ref{Thm_local ell}, we obtain that the density $R_N$ when $N \to \infty$ for $p$ in the bulk along the real line is given by 
\begin{equation}
R_{N}(x+iy) = 4 y F(2y) + \OO(e^{-N\eps}), \qquad F(z):=e^{-z^2} \int_0^z e^{t^2}\,dt.
\end{equation}
Here $F$ is Dawson's integral function.
We mention that the limiting $1$-point function $4y F(2y)$ for the bulk case is balanced in a sense that
\begin{equation} \label{bulk limit balanced}
\int_\R (4y F(2y)-1) \,dy=0. 
\end{equation}
(This easily follows from $ F(2y)/2= \int 1- 4y F(2y) \,dy. $)
The relation \eqref{bulk limit balanced} has a physical interpretation: the depletion of eigenvalues that happens close to the real axis is exactly compensated by the excess of eigenvalues in the hill, see Figure~\ref{Fig. Bulkbalanced}.

\begin{figure}[h!]
		\begin{center}	
			\includegraphics[width=0.5\textwidth]{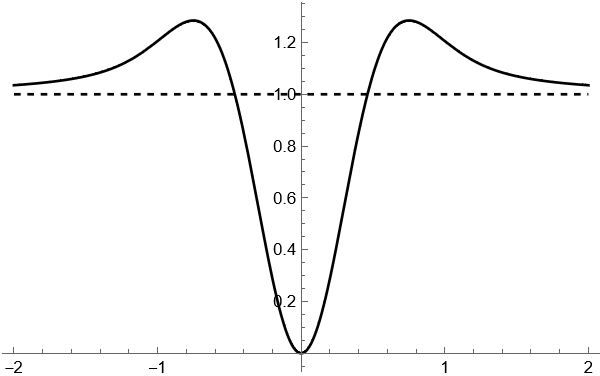}
		\end{center}
	\caption{ Graph of the function $4y F(2y)$ (full line) and its comparison with the constant function $1$ (dashed line). } \label{Fig. Bulkbalanced} 
\end{figure}

For $p$ at the real edge we have
\begin{equation}
R_N(z) = R(z) + \tfrac{1}{\sqrt{N}} R^{(1/2)}(z) + \OO(N^{-1+\eps}),
\end{equation}
where $\eps > 0$ is arbitrary,
\begin{equation} \label{R edge}
R(x+iy) = -2y\int_{-\infty}^{0}e^{-s^2}\sin(4sy)\erfc(2x-s) \, ds
\end{equation}
and 
\begin{equation}
\label{R12 edge}
 R^{(1/2)}(z) =  \tfrac{1}{3\sqrt{2}} (\tfrac{1+\tau}{1-\tau} )^{\frac32} \im(z) \, e^{-4 \im(z)^2}
\im\Big[ (2z^2+\tfrac{1-2\tau}{1+\tau})\erfc(\sqrt{2}\bar{z})\, e^{-2z^2} + 2\sqrt{\tfrac{2}{\pi}}\,\bar{z}\,e^{-4\re z^2} \Big].
\end{equation}
See Figure~\ref{Fig. Density convergence at edge} for illustrations of the edge density for finite $N$ and in the limit.

Since the first subleading term is of order $1/\sqrt{N}$, the convergence of $R_N(z)$ to the limit $R(z)$ is rather slow.
To estimate the limit from values at finite $N$ we use the following method: for given $z \in \C$ we first compute $R_N(z)$ for $N = 2000, 3000, 4000, 5000$ numerically, then we do a least squares fit of the function $N \mapsto a + b / \sqrt{N} + c / N$ to this data to estimate the constants $a, b, c$.
Since we know that the density has the series representation $R_N(z) = R(z) + R^{(1/2)}(z)/\sqrt{N} + R^{(1)}(z)/N + \cdots$ the fitted value of $a$ is our best guess for the limit $R(z)$.
We apply the same method to estimate the limit of the first correction $\sqrt{N}(R_N(z)-R(z))$.

We remark that for each $x$, as $y \to \infty,$
\begin{equation*}
R(x+iy) \sim \tfrac{1}{2}\erfc(2x), \qquad R^{(1/2)}(x+iy) \sim  -\tfrac{1}{\sqrt{\pi}}   (\tfrac{1+\tau}{1-\tau})^{ \frac32 } e^{-4x^2} y^2 .
\end{equation*}

\subsection*{Organisation of the paper} The rest of this paper is organised as follows. 
Section~\ref{sec:Kernel ODE} is devoted to recalling the skew-orthogonal polynomial formalism and to proving Proposition~\ref{Prop_ODE tau}.
In Section~\ref{sec:Strong Limit}, we perform the asymptotic analysis of the inhomogeneous terms in \eqref{ODE tau} (Proposition~\ref{Prop_rN}) and complete the proof of Theorem~\ref{Thm_local ell}. In Appendix~\ref{Appendix}, we show some preliminary estimates used in Section~\ref{sec:Strong Limit}.

\section{Generalised Christoffel-Darboux identity} \label{sec:Kernel ODE}

In this section, we derive a version of the Christoffel-Darboux formula.
For this we will use an explicit expression of the pre-kernels in terms of the Hermite polynomials.

\subsection{Skew-orthogonal polynomials} \label{Subsec_SOP}
Following \cite{MR1928853} we express the pre-kernel, and by extension all correlation functions, in terms of skew-orthogonal polynomials.
Define the skew-symmetric form
\begin{equation*}
\langle f, g \rangle_s := \int_{\C} \Big( f(\zeta) g(\bar{\zeta}) - g(\zeta) f(\bar{\zeta}) \Big) (\zeta - \bar{\zeta}) e^{-N Q(\zeta)} \,dA(\zeta).
\end{equation*}
Let $q_m$ be a family of monic polynomials of degree $m$ that satisfy the following skew-orthogonality conditions with positive skew-norms $r_k > 0$: for all $k, l \in \N$
\begin{equation}
\langle q_{2k}, q_{2l} \rangle_s = \langle q_{2k+1}, q_{2l+1} \rangle_s = 0, \qquad \langle q_{2k}, q_{2l+1} \rangle_s = -\langle q_{2l+1}, q_{2k} \rangle_s = r_k  \,\delta_{k, l}.
\end{equation}
Here $\delta_{k, l}$ is the Kronecker delta. 
Such polynomials exist if $\int \abs{\zeta}^m e^{-N Q(\zeta)}\, dA(\zeta) < \infty$ for all $m \in \N$, see e.g.\ \cite{akemann2021skew}.
Then the pre-kernel $\bfkappa_N$ is written as 
\begin{equation}\label{bfkappaN skewOP}
\bfkappa_N(\zeta,\eta)=\sum_{k=0}^{N-1} \frac{q_{2k+1}(\zeta) q_{2k}(\eta) -q_{2k}(\zeta) q_{2k+1}(\eta)}{r_k}.
\end{equation}

For the elliptic potential $Q$ in \eqref{Q elliptic}, the associated skew-orthogonal polynomials $q_k$ and their skew-norms $r_k$ are given by 
$$
q_{2k}(\zeta) = \Big( \frac{2}{N} \Big)^k k! \sum_{l=0}^{k}  \frac{(\tau/2)^l}{(2l)!!} H_{2l}  \Big( \sqrt{ \tfrac{N}{2\tau} } \zeta \Big), \qquad q_{2k+1}(\zeta) = \Big( \frac{\tau}{2N} \Big)^{k+\frac12} H_{2k+1}\Big( \sqrt{ \tfrac{N}{2\tau} } \zeta \Big),
$$
and 
$$
r_k = 2(1-\tau)^{3/2}(1+\tau)^{1/2}\frac{(2k+1)!}{N^{2k+2}},
$$
see \cite{MR1928853}.  
Thus the pre-kernel $\bfkappa_N$ of the ensemble $\boldsymbol{\zeta}$ has the expression
\begin{equation*}
\begin{split}
\bfkappa_N(\zeta,\eta)
&= \Big( \frac{N}{2(1-\tau)} \Big)^{ \frac{3}{2} } \sqrt{ \frac{\tau}{1+\tau} } \sum_{k=0}^{N-1}  \frac{ (\tau/2)^k }{(2k+1)!!} \sum_{l=0}^k  \frac{(\tau/2)^l}{(2l)!!} \\
&\quad \times \Big[H_{2k+1} \Big( \sqrt{ \tfrac{N}{2\tau} } \zeta \Big) H_{2l}  \Big( \sqrt{ \tfrac{N}{2\tau} } \eta \Big) - H_{2k+1} \Big( \sqrt{ \tfrac{N}{2\tau} } \eta \Big) H_{2l}  \Big( \sqrt{ \tfrac{N}{2\tau} } \zeta \Big) \Big].
\end{split}
\end{equation*}
Therefore by \eqref{kappa rescaling}, the rescaled pre-kernel $\kappa_N$ at the point $p \in \R$ is given by
\begin{align} \label{kappa F}
\kappa_N(z,w) = \sqrt{2}(1+\tau)  \Big( F(z,w)-F(w,z) \Big),
\end{align}
where 
\begin{equation}  \label{F in kappa}
F(z,w)
= \sum_{k=0}^{N-1}  \frac{ (\tau/2)^{k+\frac12} }{(2k+1)!!}   H_{2k+1} \Big( \sqrt{ \tfrac{N}{2\tau} } p+\sqrt{\tfrac{1-\tau^2}{\tau}}z  \Big)  
 \sum_{l=0}^k  \frac{(\tau/2)^l}{(2l)!!} H_{2l}  \Big( \sqrt{ \tfrac{N}{2\tau} } p+\sqrt{\tfrac{1-\tau^2}{\tau}}w \Big).
\end{equation}

\subsection{Derivation of a differential equation}

With the expression \eqref{F in kappa}, we show the Christoffel-Darboux identity (Proposition~\ref{Prop_ODE tau}) for the kernel. 

\begin{proof}[Proof of Proposition~\ref{Prop_ODE tau}] 
Throughout the proof, let us write 
\begin{equation} \label{zeta eta}
\zeta=  \sqrt{ \tfrac{N}{2\tau} } p+\sqrt{\tfrac{1-\tau^2}{\tau}}z, \qquad \eta= \sqrt{ \tfrac{N}{2\tau} } p+\sqrt{\tfrac{1-\tau^2}{\tau}}w
\end{equation}
to lighten notations. 
It is well known that the Hermite polynomial $H_j$ satisfies the following  differentiation rule and the  three term recurrence relation:  
\begin{equation} \label{three term}
H_j'(z)=2j H_{j-1}(z), \qquad 
 H_{j+1}(z)=2zH_j(z)-H_j'(z).
\end{equation}
Using these formulas, we have 
\begin{align*}
&\quad \partial_z \sum_{k=0}^{N-1}  \frac{ (\tau/2)^{k+\frac12} }{(2k+1)!!}   H_{2k+1} ( \zeta )  
\sum_{l=0}^k  \frac{(\tau/2)^l}{(2l)!!} H_{2l}  ( \eta )
\\
&= 2 \sqrt{\frac{1-\tau^2}{\tau}}  \sum_{k=0}^{N-1}  \frac{ (\tau/2)^{k+\frac12} }{(2k-1)!!}   H_{2k} ( \zeta )  
\sum_{l=0}^k  \frac{(\tau/2)^l}{(2l)!!} H_{2l}  ( \eta )
\\
&= \sqrt{\tau(1-\tau^2)}  \sum_{k=1}^{N-1}  \frac{ (\tau/2)^{k-\frac12} }{(2k-1)!!}   H_{2k} ( \zeta )  
\sum_{l=0}^k  \frac{(\tau/2)^l}{(2l)!!} H_{2l}  ( \eta )+\sqrt{2(1-\tau^2)}
\\
&= \sqrt{\tau(1-\tau^2)}  \sum_{k=1}^{N-1}  \frac{ (\tau/2)^{k-\frac12} }{(2k-1)!!}   ( 2 \zeta H_{2k-1}(\zeta)-H_{2k-1}'(\zeta) )
\sum_{l=0}^k  \frac{(\tau/2)^l}{(2l)!!} H_{2l}  ( \eta )+\sqrt{2(1-\tau^2)}.
\end{align*}
Thus we obtain 
\begin{align*}
&\quad \partial_z \sum_{k=0}^{N-1}  \frac{ (\tau/2)^{k+\frac12} }{(2k+1)!!}   H_{2k+1} ( \zeta ) \sum_{l=0}^k  \frac{(\tau/2)^l}{(2l)!!} H_{2l}  ( \eta )
\\
&=   2 \sqrt{\tau(1-\tau^2)} \, \zeta \sum_{k=0}^{N-2}  \frac{ (\tau/2)^{k+\frac12} }{(2k+1)!!}   H_{2k+1}(\zeta) \sum_{l=0}^{k+1}  \frac{(\tau/2)^l}{(2l)!!} H_{2l}  ( \eta )
\\
&\quad -\tau \partial_z \sum_{k=0}^{N-2}  \frac{ (\tau/2)^{k+\frac12} }{(2k+1)!!}  H_{2k+1}(\zeta) \sum_{l=0}^{k+1}  \frac{(\tau/2)^l}{(2l)!!} H_{2l}  ( \eta )+\sqrt{2(1-\tau^2)}.
\end{align*}

By rearranging the terms, 
\begin{align*}
&\quad \sum_{k=0}^{N-2}  \frac{ (\tau/2)^{k+\frac12} }{(2k+1)!!}  H_{2k+1}(\zeta) \sum_{l=0}^{k+1}  \frac{(\tau/2)^l}{(2l)!!} H_{2l}  ( \eta )
\\
&= F(z,w)
-  \frac{ (\tau/2)^{N-\frac12} }{(2N-1)!!}  H_{2N-1}(\zeta) \sum_{l=0}^{N-1}  \frac{(\tau/2)^l}{(2l)!!} H_{2l}  ( \eta )
 + \sum_{k=0}^{N-2}  \frac{ (\tau/2)^{2k+\frac32} }{(2k+2)!}  H_{2k+1}(\zeta)    H_{2k+2}  ( \eta ).
\end{align*}
Therefore we obtain that 
\begin{align*}
\partial_z F(z,w)&= 2 \sqrt{\tau(1-\tau^2)} \, \zeta  F(z,w) -\tau \partial_z F(z,w)
\\
&\quad -   \sqrt{\tau(1-\tau^2)}  \frac{ (\tau/2)^{N-\frac12} }{(2N-1)!!} \Big( 2\zeta H_{2N-1}(\zeta)-H_{2N-1}'(\zeta) \Big)  \sum_{l=0}^{N-1}  \frac{(\tau/2)^l}{(2l)!!} H_{2l}  ( \eta )
\\
&\quad +   \sqrt{\tau(1-\tau^2)}  \sum_{k=0}^{N-2}  \frac{ (\tau/2)^{2k+\frac32} }{(2k+2)!}  \Big( 2\zeta H_{2k+1}(\zeta) - H'_{2k+1}(\zeta) \Big)    H_{2k+2}  ( \eta )+\sqrt{2(1-\tau^2)}.
\end{align*}
We now use the three term recurrence relation again and obtain
\begin{align*}
&\quad \partial_z F(z,w)= 2 \sqrt{\tau(1-\tau^2)} \, \zeta  F(z,w) -\tau \partial_z F(z,w)
\\
&-   \sqrt{2(1-\tau^2)} \Big[ \frac{ (\tau/2)^{N} }{(2N-1)!!} H_{2N}(\zeta)  \sum_{l=0}^{N-1}  \frac{(\tau/2)^l}{(2l)!!} H_{2l}  ( \eta )
-  \sqrt{2(1-\tau^2)}  \sum_{k=0}^{N-1}  \frac{ (\tau/2)^{2k} }{(2k)!}  H_{2k}(\zeta)    H_{2k}  ( \eta ) \Big].
\end{align*}
This leads to 
\begin{align}
\begin{split}
&\quad \partial_z F(z,w)=\Big( \sqrt{ \frac{2(1-\tau) N}{1+\tau} } p+2(1-\tau)z\Big)  F(z,w)
\\
&-  \frac{ \sqrt{2(1-\tau^2)} }{1+\tau} \Big[ \frac{ (\tau/2)^{N} }{(2N-1)!!} H_{2N}(\zeta)  \sum_{l=0}^{N-1}  \frac{(\tau/2)^l}{(2l)!!} H_{2l}  ( \eta )
- \frac{ \sqrt{2(1-\tau^2)} }{ 1+\tau }  \sum_{k=0}^{N-1}  \frac{ (\tau/2)^{2k} }{(2k)!}  H_{2k}(\zeta)    H_{2k}  ( \eta ) \Big].
\end{split}
\end{align}

Similarly, we obtain 
\begin{align*}
&\quad \partial_z \sum_{k=0}^{N-1}  \frac{ (\tau/2)^{k+\frac12} }{(2k+1)!!}   H_{2k+1} ( \eta )  
\sum_{l=0}^k  \frac{(\tau/2)^l}{(2l)!!} H_{2l}( \zeta )
\\
&= 2 \sqrt{\frac{1-\tau^2}{\tau}}  \sum_{k=0}^{N-1}  \frac{ (\tau/2)^{k+\frac12} }{(2k+1)!!}   H_{2k+1} ( \eta )  
\sum_{l=1}^k  \frac{(\tau/2)^l}{(2l-2)!!} H_{2l-1} ( \zeta )
\\
&= 2 \sqrt{\frac{1-\tau^2}{\tau}}  \sum_{k=0}^{N-1}  \frac{ (\tau/2)^{k+\frac12} }{(2k+1)!!}   H_{2k+1} ( \eta )  
\sum_{l=1}^k  \frac{(\tau/2)^l}{(2l-2)!!} \Big(  2\zeta H_{2l-2}(\zeta)-H_{2l-2}'(\zeta) \Big)
\\
&= \sqrt{(1-\tau^2) \tau}  \sum_{k=0}^{N-1}  \frac{ (\tau/2)^{k+\frac12} }{(2k+1)!!}   H_{2k+1} ( \eta )  
\sum_{l=0}^{k-1}  \frac{(\tau/2)^{l}  }{(2l)!!} \Big(  2\zeta H_{2l}(\zeta)-H_{2l}'(\zeta) \Big).
\end{align*}
Thus we have
\begin{align*}
\partial_z F(w,z) 
&=   \sqrt{(1-\tau^2)\tau}  \sum_{k=0}^{N-1}  \frac{ (\tau/2)^{k+\frac12} }{(2k+1)!!}   H_{2k+1} ( \eta )  
\sum_{l=0}^{k}  \frac{(\tau/2)^{l}  }{(2l)!!} \Big(  2\zeta H_{2l}(\zeta)-H_{2l}'(\zeta) \Big)
\\
&\quad -   \sqrt{2(1-\tau^2)}  \sum_{k=0}^{N-1}  \frac{ (\tau/2)^{2k+1} }{(2k+1)!}    H_{2k+1}(\zeta) H_{2k+1} ( \eta ) .
\end{align*}
Therefore we obtain 
\begin{align*}
\partial_z F(w,z)&= \Big( \sqrt{ 2(1-\tau^2) N } p+2(1-\tau^2)z\Big) F(w,z)
- \tau \partial_z F(w,z)
\\
&\quad -   \sqrt{2(1-\tau^2)}  \sum_{k=0}^{N-1}  \frac{ (\tau/2)^{2k+1} }{(2k+1)!}    H_{2k+1}(\zeta) H_{2k+1} ( \eta ) ,
\end{align*}
which leads to 
\begin{align}
\begin{split}
\partial_z F(w,z) &=\Big( \sqrt{ \frac{2(1-\tau) N}{1+\tau} } p+2(1-\tau)z\Big)  F(w,z) 
- \frac{ \sqrt{2(1-\tau^2)} }{ 1+\tau }    \sum_{k=0}^{N-1}  \frac{ (\tau/2)^{2k+1} }{(2k+1)!}    H_{2k+1}(\zeta) H_{2k+1} ( \eta ).
\end{split}
\end{align}

Combining all of the above, we conclude
\begin{align*}
 \partial_z \Big( F(z,w)-F(w,z) \Big) &= \Big( \sqrt{ \frac{2(1-\tau) N}{1+\tau} } p+2(1-\tau)z\Big)  \Big( F(z,w)-F(w,z) \Big) 
 \\
 &\quad +  \frac{ \sqrt{2(1-\tau^2)} }{ 1+\tau }  \sum_{k=0}^{2N-1}  \frac{ (\tau/2)^{k} }{k!}  H_{k}(\zeta)    H_{k}  ( \eta )
 \\
 &\quad -  \frac{ \sqrt{2(1-\tau^2)} }{1+\tau} \frac{ (\tau/2)^{N} }{(2N-1)!!} H_{2N}(\zeta)  \sum_{l=0}^{N-1}  \frac{(\tau/2)^l}{(2l)!!} H_{2l}  ( \eta ).
\end{align*}
Now Proposition~\ref{Prop_ODE tau} follows from the relation \eqref{kappa F}.

\end{proof}

To analyse the large-$N$ limit of the pre-kernel, let us introduce 
\begin{equation} \label{transformed kernel}
\widehat{\kappa}_N(z,w) := \omega_N(z,w) \kappa_N(z,w),
\end{equation}
where 
\begin{equation}
\omega_N(z,w) = \exp\Big[ \tau \, \Big( p \sqrt{\tfrac{N}{2(1-\tau^2)}}+z \Big)^2 +  \tau \, \Big( p \sqrt{\tfrac{N}{2(1-\tau^2)}}+w \Big)^2 - \Big( p\sqrt{\tfrac{N}{1-\tau^2}}+\sqrt{2}z \Big) \Big( p\sqrt{\tfrac{N}{1-\tau^2}}+\sqrt{2}w \Big) \Big].
\end{equation}

Let us write 
\begin{align}
\begin{split}
E^1_N(\xi, \omega) &:= 2 \sqrt{1-\tau^2} \exp\Big[  N \Big(  \tfrac{\tau}{2} (\xi^2 + \omega^2)- \xi \omega   \Big)  \Big]  \\
&\quad \times \sum_{k=0}^{2N-1} \frac{ (\tau/2)^{k} }{k!} H_{k}\Big( \sqrt{N \tfrac{1-\tau^2}{2\tau}} \xi \Big) H_{k}\Big( \sqrt{N \tfrac{1-\tau^2}{2\tau}} \omega \Big), \label{E_N^1}
\end{split}
\end{align}
and 
\begin{align}
\begin{split}
E^2_N(\xi, \omega) &:= 2 \sqrt{1-\tau^2} \exp\Big[ -N \tfrac{1-\tau}{2} (\xi^2 + \omega^2) \Big] 
\\
&\quad \times \frac{ (\tau/2)^{N} }{(2N-1)!!}  H_{2N}\Big( \sqrt{N \tfrac{1-\tau^2}{2\tau}} \xi \Big)  \sum_{l=0}^{N-1}  \frac{(\tau/2)^l}{(2l)!!} H_{2l}\Big( \sqrt{N \tfrac{1-\tau^2}{2\tau}} \omega \Big). \label{E_N^2}
\end{split}
\end{align}
Then as an immediate consequence of Proposition~\ref{Prop_ODE tau}, we have the following corollary. 

\begin{cor} \label{Cor_ODE transform}
We have 
\begin{equation} \label{ODE tau transformed with error terms}
\partial_z \widehat{\kappa}_N(z,w) = 2(z-w)\widehat{\kappa}_N(z,w) + r_N(z,w),
\end{equation}
where 
\begin{equation}
r_N(z,w) = E_N^1\Big( \tfrac{p}{\sqrt{1-\tau^2}} + \sqrt{\tfrac{2}{N}}z, \tfrac{p}{\sqrt{1-\tau^2}} + \sqrt{\tfrac{2}{N}}w \Big)
-e^{(z-w)^2} E_N^2\Big( \tfrac{p}{\sqrt{1-\tau^2}} + \sqrt{\tfrac{2}{N}}z, \tfrac{p}{\sqrt{1-\tau^2}} + \sqrt{\tfrac{2}{N}}w \Big). \label{rN} 
\end{equation}
\end{cor}

The asymptotic expansion of $r_N$ will be addressed in the next section.

\section{Asymptotic analysis for the correlation kernels} \label{sec:Strong Limit}

In this section, we analyse the large-$N$ asymptotics of the differential equation \eqref{ODE tau transformed with error terms} and prove Theorem~\ref{Thm_local ell}.

\subsection{Asymptotics of the inhomogeneous terms}

As indicated by Lemma~\ref{Lem_ODE rate} below, the key ingredient to prove Theorem~\ref{Thm_local ell} is the large-$N$ behaviour of the function $r_N$.  

\begin{prop}\label{Prop_rN}
The following asymptotics hold as $N \to \infty$.
\begin{itemize}
    \item \textbf{\textup{(The real bulk)}} For $|p|<\sqrt{2}(1+\tau)$, there exists an $\eps>0$ such that
    \begin{equation*}
    r_N(z,w) = 2 + \OO(e^{-N\eps}).
    \end{equation*}

    \item \textbf{\textup{(The real edge)}} For $p = \pm \sqrt{2}(1+\tau)$ and an arbitrary $\eps > 0$, it holds
    \begin{equation}
    r_N(z,w)=r(z,w)+\tfrac{1}{\sqrt{N}} r^{(1/2)}(z,w)+\OO(N^{-1+\eps}), 
    \end{equation}
    where 
    \begin{equation}
    r(z,w):=  \erfc(z+w) -\tfrac{1}{\sqrt{2}}e^{(z-w)^2-2z^2}\erfc(\sqrt{2}w)
    \end{equation}
    and 
    \begin{align}
    \begin{split}
    r^{(1/2)}(z,w)&:= \tfrac{1}{\sqrt{2}}(\tfrac{1+\tau}{1-\tau})^{\frac32} e^{(z-w)^2-2z^2} 
    \\
    &\quad \times \Big[\tfrac{1}{\sqrt{2\pi}} (\tfrac{4}{3}z^2-\tfrac{4}{3}zw+\tfrac{2}{3}w^2-\tfrac{\tau}{1+\tau}) e^{-2w^2} -(\tfrac{2}{3}z^3-\tfrac{\tau}{1+\tau}z)\erfc(\sqrt{2}w)\Big].
    \end{split}
    \end{align}
    
    \item \textbf{\textup{(The real axis outside the droplet)}}  For $|p|>\sqrt{2}(1+\tau)$, there exists an $\eps>0$ such that
    \begin{equation}
    r_N(z) = \OO(e^{-N\eps}).
    \end{equation}
\end{itemize}
In all three cases the error terms are uniform for $z, w$ in compact subsets of $\C$. 
\end{prop}

Combining Corollary~\ref{Cor_ODE transform} and Proposition~\ref{Prop_rN}, our main result Theorem~\ref{Thm_local ell} can be easily derived by solving the associated differential equations, see Subsection~\ref{Subsec_main thm proof}.
The proof of each statement in Proposition~\ref{Prop_rN} is given in the following subsections. 

\begin{rmk*}
For the special case when $p=0$, the limiting kernel was obtained in \cite[Theorem 4.1]{akemann2021skew} using an integral representation of the Hermite polynomials. 
From the viewpoint of Proposition~\ref{Prop_rN}, it also follows from the convergence $\lim_{N \to \infty} r_N=2$.
We remark that this convergence can be alternatively obtained using the classical Mehler-Hermite formula (see e.g.\ \cite[Eq.(18.18.28)]{olver2010nist}).
\end{rmk*}

For the reader's convenience, let us outline the strategy for the proof of Proposition~\ref{Prop_rN}. 
\begin{itemize}
    \item We rewrite $E_N^1, E_N^2$ as
\begin{equation} \label{eq:E_N^12 integral decomposition}
E_N^1(\xi, \omega) = E_N^1(\xi_0, \omega) + \int_{\xi_0}^{\xi} \partial_\xi E_N^1(t, \omega) \, dt , \qquad
E_N^2(\xi, \omega) = E_N^2(\xi, \omega_0) + \int_{\omega_0}^{\omega} \partial_\omega E_N^2(\xi, t) \, dt.
\end{equation}
\item Using a version of the Christoffel-Darboux identity, we express $\partial_\xi E_N^1$ and $\partial_\omega E_N^2$ only in terms of a few orthogonal polynomials (Lemma~\ref{Lem_ EN1 EN2 deri}). 
\smallskip 
\item Using the strong asymptotics of the Hermite polynomials (Lemma~\ref{Lem_SA Hermite}), we estimate each term in \eqref{eq:E_N^12 integral decomposition} for suitable choices of $\xi_0$ and $\omega_0$. 
\end{itemize}

We end this subsection by introducing the main ingredients for the proof of Proposition~\ref{Prop_rN}.

\begin{lem} \label{Lem_ EN1 EN2 deri}
For each $N$, we have
\begin{align}
\label{Derivative E_N^1}
\begin{split}
\pa_\xi E_N^1(\xi, \omega) &= \sqrt{2\tau} \Big( \frac{\tau}{2} \Big)^{2N-1} \frac{\sqrt{N}}{(2N-1)!}  
\exp\Big[ N \Big( \tfrac{\tau}{2} (\xi^2 + \omega^2) - \xi\omega \Big) \Big] 
\\
&\quad \times \Big[ \tau H_{2N}\Big( \sqrt{N\tfrac{1-\tau^2}{2\tau}} \xi \Big) H_{2N-1}\Big( \sqrt{N\tfrac{1-\tau^2}{2\tau}} \omega \Big) - H_{2N-1}\Big( \sqrt{N\tfrac{1-\tau^2}{2\tau}} \xi \Big) H_{2N}\Big( \sqrt{N\tfrac{1-\tau^2}{2\tau}} \omega \Big) \Big]
\end{split}
\end{align}
and 
\begin{align} \label{Derivative E_N^2}
\begin{split}
\pa_\omega E_N^2(\xi, \omega) &= - \sqrt{2\tau} (1-\tau) \Big( \frac{\tau}{2} \Big)^{2N-1} \frac{\sqrt{N}}{(2N-1)!} \exp\Big( -N\tfrac{1-\tau}{2} (\xi^2+\omega^2) \Big) 
\\
&\quad \times H_{2N}\Big( \sqrt{N\tfrac{1-\tau^2}{2\tau}} \xi \Big) H_{2N-1}\Big( \sqrt{N\tfrac{1-\tau^2}{2\tau}} \omega \Big).
\end{split}
\end{align}
\end{lem}
\begin{proof}
This lemma immediately follows from the Christoffel-Darboux identity for the kernel of complex elliptic Ginibre ensembles found in \cite[Proposition 2.3]{MR3450566}. (See \cite[Section 3]{byun2021lemniscate} for the derivation of such identities in a more general setup.) 

To be more precise, let
\begin{equation}
S_N(\zeta, \eta) := \sum_{k=0}^{N-1} \frac{(\tau/2)^k}{k!} H_k(\zeta) H_k(\eta).
\end{equation}
Then by \eqref{E_N^1} and \eqref{E_N^2}, one can express $E_N^1$ and $E_N^2$ in terms of $S_N$. In particular, for $E_N^2$, we use the realisation 
\begin{equation*}
\sum_{l=0}^{N-1} \frac{(\tau/2)^l}{(2l)!!} H_{2l}\Big( \sqrt{N\tfrac{1-\tau^2}{2\tau}} \omega \Big)
= \sum_{k=0}^{2N-1} \frac{(i\sqrt{\tau}/2)^k}{k!} H_{k}\Big( \sqrt{N\tfrac{1-\tau^2}{2\tau}} \omega \Big) H_{k}(0),
\end{equation*}
which follows from $H_{2l}(0) = (-1)^l (2l)!/l!$ and $H_{2l+1}(0) = 0$.
Now the Christoffel-Darboux identity 
\begin{equation} \label{Derivative S_N}
\begin{split}
\pa_\zeta S_N(\zeta, \eta) &= \frac{2\tau}{1-\tau^2} (\eta-\tau\zeta) S_N(\zeta, \eta) + \frac{2}{1-\tau^2} \Big( \frac{\tau}{2} \Big)^{N} \frac{\tau H_{N}(\zeta) H_{N-1}(\eta) - H_{N-1}(\zeta) H_{N}(\eta)}{(N-1)!}
\end{split}
\end{equation}
completes the proof. 
\end{proof}

The strong asymptotics of the Hermite polynomials can be found for instance in \cite{MR3450566,MR3274226}. 
(See also \cite{MR3043718,MR1677884} for the strong asymptotics of the classical orthogonal polynomials.) 
Here we follow the conventions in \cite{MR3450566}.
To describe such asymptotic behaviours, for $T > 0$, let 
\begin{equation} \label{l LR}
l := T \Big( \log  \frac{T}{1-\tau^2}  - 1 \Big)
\end{equation}
and
\begin{equation} \label{g function LR}
g(z) := \log \Big( z+\sqrt{z^2-F_0^2} \Big) + \frac{z}{  z+\sqrt{z^2-F_0^2} } - \log 2 - \frac{1}{2}, \qquad F_0 := 2\sqrt{T\tfrac{\tau}{1-\tau^2}}.
\end{equation}
We also write 
\begin{equation} \label{psi LR}
\psi(z) := \frac{\sqrt{\tau}}{F_0} \, \Big( z+\sqrt{z^2-F_0^2} \Big).
\end{equation}

\begin{lem} \label{Lem_SA Hermite} \textup{(Strong asymptotics of Hermite polynomials)}
For $T > 0$ and $R \in \Z$, we have 
\begin{equation}
\Big| H_{TN+R}\Big( \sqrt{N\tfrac{1-\tau^2}{2\tau}} z \Big) \Big| = \Big( \frac{\tau}{2} \Big)^{-\frac{TN+R}{2}} \, \sqrt{\frac{(TN+R)!}{N}} \, \Big| e^{N (Tg(z) - l/2)}\Big| \, \OO(N^{\frac{5}{12}}), \label{Hermite strong asymptotic bound} 
\end{equation}
where the $\OO(N^{\frac{5}{12}})$-term is uniform for $z \in \mathbb{C}$, and 
\begin{align}
\begin{split}
H_{TN+R}\Big( \sqrt{N\tfrac{1-\tau^2}{2\tau}} z \Big) &= \Big( \frac{N}{2\pi(1-\tau^2)} \Big)^{\frac{1}{4}} \Big( \frac{\tau}{2} \Big)^{-\frac{TN+R}{2}} \, \sqrt{\frac{(TN+R)!}{N}}
\\
&\quad \times \psi(z)^R \sqrt{\psi^{\prime}(z)} e^{N (Tg(z) - l/2)} \big( 1 + \OO(1/N) \big), \label{Hermite strong asymptotic expansion}
\end{split}
\end{align}
where the $\OO(1/N)$-term is uniform for any compact subset of $\C \setminus [-F_0, F_0]$.
\end{lem}

In the sequel, it is convenient to write 
\begin{equation}
K_T := \{ z\in\C \, | \, \tfrac{1-\tau}{1+\tau} (\re z)^2 + \tfrac{1+\tau}{1-\tau} (\im z)^2 \leq T  \}, \qquad
\curvature_T := \tfrac{1}{\sqrt{T}} ( \tfrac{1+\tau}{1-\tau} )^{\frac32}, \label{KT LR}
\end{equation}
and
\begin{equation}
\Omega(z) := \abs{z}^2 - \tau\re z^2 - 2T\re g(z) + l. \label{Omega LR}
\end{equation}
In the rest of this section, we will mainly consider the case $T = 2$.

\subsection{Inside the bulk}

In this section we prove the first part of Proposition~\ref{Prop_rN}. By \eqref{rN}, it immediately follows from the next lemma.

\begin{lem} \label{Error terms in real bulk}
Let $x_0 \in \R$ with $\abs{x_0} < \sqrt{2 \frac{1+\tau}{1-\tau}}$. Then there exist a neighbourhood $U \subset \C$ of $x_0$ and an $\eps > 0$ such that the following estimates hold uniformly for $\xi, \omega \in U$:
\begin{equation*}
E_N^1(\xi, \omega) = 2 + \OO(e^{-N\eps}), \qquad
E_N^2(\xi, \omega) = \OO(e^{-N\eps}).
\end{equation*}
\end{lem}

\begin{proof}

By Lemma~\ref{Lem_ EN1 EN2 deri} and the asymptotic \eqref{Hermite strong asymptotic bound}, there exists a constant $C > 0$ such that
\begin{align*}
\abs{\pa_\xi E_N^1(\xi, \omega)} &\leq C N^{5/6} \exp\Big( -N\re \Big[ \xi\omega - \tfrac{\tau}{2}(\xi^2+\omega^2) - 2g(\xi) - 2g(\omega) + l \Big] \Big), \\
\abs{\pa_\omega E_N^2(\xi, \omega)} &\leq C N^{5/6} \exp\Big( -N\re \Big[ \xi\omega - \tfrac{\tau}{2}(\xi^2+\omega^2) + \tfrac{(\xi-\omega)^2}{2}- 2g(\xi) - 2g(\omega) + l \Big] \Big) 
\end{align*}
for a sufficiently large $N$ and $\xi, \omega \in \C$.

By \cite[Lemma D.1]{MR3450566}, we have that $\Omega(z) \geq 0$ for all $z \in \C$ and that $\Omega(z) = 0$ only for $z \in \partial K_T$. 
Here, $\Omega$ is given by \eqref{Omega LR}.
Notice that $x_0 \in \mathring{K}_2$ (i.e.\ \eqref{KT LR} with $T=2$).
Hence there exists a compact set $V \subset \mathring{K}_2$ that contains an open neighbourhood of $x_0$ and $\Omega(z) > 2\eps$ for all $z \in V$ with some $\eps>0$.
Consequently we obtain that for $\xi, \omega \in V$ with $\abs{\xi-\bar{\omega}} < \sqrt{2\eps}$,
\begin{align*}
\re \Big[ \xi\omega - \tfrac{\tau}{2}(\xi^2+\omega^2) - 2g(\xi) - 2g(\omega) + l \Big]
= -\frac{1}{2}\abs{\xi-\bar{\omega}}^2 + \frac{1}{2}\Omega(\xi) + \frac{1}{2}\Omega(\omega)
> \eps.
\end{align*}
Let us take 
$U^1 \subset V \cap \{ z \mid \abs{z-x_0} < \sqrt{\eps/2} \}.$ 
Then we have the uniform estimate
\begin{equation} \label{Derivative E_N^1 Estimate}
\partial_\xi E_N^1(t, \omega) = \OO(N^{5/6} e^{-N\eps}), \qquad (t, \omega \in U^1).
\end{equation}

To analyse $E_N^2$, notice that for $\xi, \omega \in V$ with $(\im \xi)^2, (\im \omega)^2 < \eps/2$,
\begin{align*}
\re \Big[ \xi\omega - \tfrac{\tau}{2}(\xi^2+\omega^2) + \tfrac{(\xi-\omega)^2}{2}- 2g(\xi) - 2g(\omega) + l \Big]
= - (\im \xi)^2 - (\im \omega)^2  + \frac{1}{2}\Omega(\xi) + \frac{1}{2}\Omega(\omega)
> \eps .
\end{align*}
Let $U^2 \subset V \cap \{ z \mid \abs{\im z} < \sqrt{\eps/2} \}$. 
Then 
\begin{equation} \label{Derivative E_N^2 Estimate}
\partial_\omega E_N^2(\xi, t) = \OO(N^{5/6} e^{-N\eps}), \qquad (\xi, t \in U^2) 
\end{equation}
holds uniformly. 

Next, we analyse $E_N^1(\xi_0, \omega)$ and $E_N^2(\xi, \omega_0)$ for suitable choices of $\xi_0$ and $\omega_0$. 
For $E_N^1$, let us choose $\xi_0 = \bar{\omega}$. 
Then $E_N^1(\bar{\omega}, \omega)$ corresponds to the density of the complex elliptic Ginibre ensemble with $2N$ eigenvalues in the bulk. 
(See the next subsection for further details.)
In particular, by \cite[Lemma D.4]{MR3450566}, we have the uniform estimate
\begin{equation} \label{E_N^1 Initial value}
E_N^1(\bar{\omega}, \omega)
= 2 + \OO(e^{-N\eps}).
\end{equation}

We now claim that
\begin{equation} \label{E_N^2 Initial value}
E_N^2(\xi, 0)
= \OO(e^{-N\eps/2} N^{ \frac{1}{6} }), \qquad (\xi \in U^2)
\end{equation}
holds uniformly. 
By \eqref{Hermite strong asymptotic bound}, we have
\begin{align*}
&\quad \Big| \exp\Big( -N\frac{1-\tau}{2} \xi^2 \Big) \frac{(\tau/2)^N}{(2N-1)!!} H_{2N}\Big( \sqrt{N\tfrac{1-\tau^2}{2\tau}} \xi \Big) \Big|
\\
&= \Big( \frac{(2N)!!}{(2N-1)!! \, N} \Big)^{\frac12} \exp\Big( -N\re\Big[ \tfrac{1-\tau}{2} \xi^2 -2g(\xi) + \tfrac{l}{2} \Big]\Big) \OO(N^{ \frac{5}{12} }).
\end{align*}
We also have that for $\xi \in V$ with $(\im \xi)^2 < \eps$,
\begin{equation*}
\re\Big[ \tfrac{1-\tau}{2} \xi^2 -2g(\xi) + \tfrac{l}{2} \Big]
= \frac{1}{2}\Omega(\xi) - \frac{1}{2}(\im \xi)^2
> \frac{\eps}{2}.
\end{equation*}
Notice also that Stirling's formula yields
\begin{equation*}
\Big( \frac{(2 N)!!}{(2 N - 1)!! \, N} \Big)^{\frac12} = \OO(N^{-\frac14}).
\end{equation*}
Finally, due to a well-known Taylor series, we have
\begin{align} \label{E_N^2 Taylor sum}
\sum_{l=0}^{N-1} \frac{(\tau/2)^l}{(2l)!!} H_{2l}(0)
= \sum_{l=0}^{N-1} \Big( -\frac{\tau}{4} \Big)^l \, \frac{(2l)!}{(l!)^2}
= \frac{1}{\sqrt{1+\tau}} + \OO\Big( \tfrac{\tau^N}{\sqrt{N}} \Big)
\end{align}
Combining all of the above with \eqref{E_N^2}, we obtain \eqref{E_N^2 Initial value}. 

We now choose $U \subset U^1 \cap U^2 = V \cap \{z \mid \abs{z-x_0} < \sqrt{\eps/2}\}$.
Then using \eqref{Derivative E_N^1 Estimate}, \eqref{E_N^1 Initial value} for $E_N^1$ and \eqref{Derivative E_N^2 Estimate}, \eqref{E_N^2 Initial value} for $E_N^2$, we obtain
\begin{align*}
E_N^1(\xi, \omega) &= E_N^1(\bar{\omega}, \omega) + \int_{\bar{\omega}}^{\xi} \partial_\xi E_N^1(t, \omega) \, dt = 2 + 
\OO(N^{ \frac{5}{6} } e^{-N\eps}), \\
E_N^2(\xi, \omega) &= E_N^2(\xi, 0) + \int_{0}^{\omega} \partial_\omega E_N^2(\xi, t) \, dt = \OO(N^{ \frac16 } e^{-\frac{N\eps}{2}}) 
,
\end{align*}
where the error terms are uniform for $\xi, \omega \in U$. This completes the proof. 
\end{proof}

\subsection{At the edge}

In this subsection, we prove the second assertion of Proposition~\ref{Prop_rN}. 

Let us define 
\begin{equation} \label{calK}
\calK_n(z, \overline{w}) := N \sqrt{1-\tau^2} e^{ -N ( z \overline{w} - \tfrac{\tau}{2} (z^2 + \overline{w}^2) ) } \sum_{j=0}^{n-1} \frac{(\tau/2)^j}{j!} H_j\Big( \sqrt{N\tfrac{1-\tau^2}{2\tau}} z \Big) \overline{ H_j\Big( \sqrt{N\tfrac{1-\tau^2}{2\tau}} w \Big) }.
\end{equation}
Notice that the function $\calK_{2N}$ is same as $E_N^1$ in \eqref{E_N^1} up to a prefactor. 

It is well known that $\calK_n$ corresponds to the pre-kernel of the complex elliptic Ginibre ensemble with $n$ points, which is known to form a determinantal point process, see e.g.\ \cite{fyodorov1997almost}.  
As $n, N \to \infty$ while keeping $n/N=T$ fixed, the elliptic Ginibre ensemble condensates to the droplet $K_T$ in \eqref{KT LR}. 
Moreover, the asymptotic behaviour of \eqref{calK} on the diagonal was extensively studied in \cite{MR3450566}. Along the same lines as the proof of \cite[Theorem 1.1]{MR3450566}, one can obtain the following proposition. 
(See also \cite[Corollary 8.12]{riser2013universality} for the leading term of the proposition.) 

\begin{prop} \textup{(Cf. \cite[Theorem 3.10]{MR3450566})} \label{LeeRiser edge kernel}
Given $T>0$ and $z_0 \in \partial K_T$, let $\mathbf{n}=\mathbf{n}(z_0)\in\C$ be the outer unit normal vector of $K_T$ at $z_0$.
Then as $n, N \to \infty$ while keeping $n/N=T$ fixed and for an arbitrary $\eps > 0$, we have
\begin{equation} \label{calK expansion}
\frac{1}{N}\calK_n\Big( z_0 + \tfrac{\xi\mathbf{n}}{\sqrt{N}}, \overline{z_0 + \tfrac{\omega\mathbf{n}}{\sqrt{N}}} \Big)
=\frac{1}{2} \erfc\Big(\frac{\xi+\overline{\omega}}{\sqrt{2}}\Big)  + \frac{1}{\sqrt{N}} \frac{\curvature(z_0)}{\sqrt{2\pi}} \frac{\xi^2-\xi\overline{\omega}+\overline{\omega}^2-1}{3} e^{-\frac12 (\xi+\overline{\omega})^2}  + \OO(N^{-1+\eps}),
\end{equation}
where $\curvature(z_0)$ is the curvature of $\partial K_T$ at $z_0$.
Here the error bound is uniform for bounded $\xi, \omega$ and over $z_0 \in \partial K_T$.
\end{prop}

We shall need the following elementary estimates. We provide the proof in Appendix~\ref{Appendix 1}.
Recall that $l, g, \psi$ and $\curvature_T$ are given by \eqref{l LR}, \eqref{g function LR}, \eqref{psi LR} and \eqref{KT LR}.

\begin{lem} \label{Estimates for E_N^2 at edge}
Let $T>0$ and $x_0=\sqrt{T\frac{1+\tau}{1-\tau}}$.
As $N \to \infty$, we have the following:
\begin{itemize}
    \item for $0\leq\alpha<1/2$, we have
    \begin{equation} \label{Lemma A1 1}
    \re\Big[ \tfrac{1-\tau}{2}s^2-Tg(s)+\tfrac{l}{2} \Big] \geq N^{-1+2\alpha}, \qquad
    s \in [0,x_0-\sqrt{\tfrac{2}{N}}N^\alpha];
    \end{equation}
    
    \item for $0\leq\nu<1/6$, we have
    \begin{equation} \label{Lemma A1 2}
    \exp\Big( -N\Big[ \tfrac{1-\tau}{2} \Big( x_0+\sqrt{\tfrac{2}{N}}z \Big)^2 - Tg\Big( x_0+\sqrt{\tfrac{2}{N}}z \Big) +\tfrac{l}{2} \Big] \Big) = e^{-2z^2} \Big( 1 + \frac{2\sqrt{2}}{3}\curvature_T\frac{z^3}{\sqrt{N}} + \OO(N^{-1+6\nu}) \Big)
    \end{equation}
    uniformly for $z=\OO(N^\nu)$;
    
    \item for $0 \leq \nu < 1/6$, we have
    \begin{align} 
    \begin{split} \label{Lemma A1 3}
    \sqrt{\psi'\Big(x_0+\sqrt{\tfrac{2}{N}}z\Big)} &= \Big(\frac{1+\tau}{T(1-\tau)}\Big)^{\frac14} \Big( 1-\frac{\sqrt{2}\tau}{1+\tau}\curvature_T\frac{z}{\sqrt{N}} + \OO(N^{-1+2\nu}) \Big), 
    \\
    \sqrt{\psi'\Big(x_0+\sqrt{\tfrac{2}{N}}z\Big)} \frac{\sqrt{\psi'\Big(x_0+\sqrt{\tfrac{2}{N}}s\Big)}}{\psi\Big(x_0+\sqrt{\tfrac{2}{N}}s\Big)} &= \Big(\frac{1+\tau}{T(1-\tau)}\Big)^{\frac12} \Big( 1-\frac{\sqrt{2}}{1+\tau}\curvature_T\frac{s+\tau z}{\sqrt{N}}+\OO(N^{-1+2\nu})\Big)
     \end{split}
    \end{align}
    uniformly for $z, s=\OO(N^\nu)$.
\end{itemize}
\end{lem}

We now prove the following lemma, which gives the second assertion of Proposition~\ref{Prop_rN}.

\begin{lem} \label{Lem_error terms at real edge}
Let $x_0 = \pm \sqrt{2 \frac{1+\tau}{1-\tau}}$. 
Then as $N \to \infty$ and for arbitrary $\eps > 0$, we have
\begin{equation}
   E_N^1\Big( x_0 \pm \sqrt{\tfrac{2}{N}} z, x_0 \pm \sqrt{\tfrac{2}{N}} w \Big)
    = \erfc(z+w) + \tfrac{\sqrt{2}\curvature_2}{\sqrt{N}} e^{-(z+w)^2} \tfrac{2z^2-2zw+2w^2-1}{3\sqrt{\pi}} + \OO(N^{-1+\eps}) \label{EN1 asym edge}
\end{equation}
and
\begin{align}
\begin{split}
    &E_N^2\Big( x_0 \pm \sqrt{\tfrac{2}{N}} z, x_0 \pm \sqrt{\tfrac{2}{N}} w \Big)= \tfrac{1}{\sqrt{2}} e^{-2z^2} \erfc(\sqrt{2}w)  
    \\
    &+ \tfrac{\curvature_2}{\sqrt{N}} e^{-2z^2} \Big[ (\tfrac{2}{3}z^3 - \tfrac{\tau}{1+\tau}z ) \erfc(\sqrt{2}w) - \tfrac{1}{\sqrt{2\pi}} (\tfrac{1}{1+\tau}-\tfrac{1+2w^2}{3}) e^{-2w^2} \Big]  
    +\OO(N^{-1+\eps}). \label{EN2 asym edge}
\end{split}
\end{align}
The error terms are uniform for bounded $z, w$.
\end{lem}

\begin{proof}

We shall consider the case $x_0 = +\sqrt{2\frac{1+\tau}{1-\tau}}$ as the other case works analogously.

The expansion \eqref{EN1 asym edge} for $E_N^1$ immediately follows from Proposition~\ref{LeeRiser edge kernel} using the relation 
$$ 
E_N^1(\xi, \omega) = \tfrac{2}{N} \calK_{2N}(\xi, \omega).
$$
Thus it suffices to show the expansion \eqref{EN2 asym edge} for $E_N^2$. 
To analyse $E_N^2$, choose $0 < \alpha < \nu < 1/6$ and let us write
\begin{align} 
\begin{split} \label{EN2 edge decom}
&\quad E_N^2\Big( x_0 + \sqrt{\tfrac{2}{N}}z, x_0 + \sqrt{\tfrac{2}{N}}w \Big) = E_N^2 \Big( x_0 + \sqrt{\tfrac{2}{N}}z, 0 \Big)
\\
&+ \int_{0}^{x_0 - \sqrt{\frac{2}{N}}N^\alpha} \pa_{\omega} E_N^2\Big( x_0 + \sqrt{\tfrac{2}{N}}z, s \Big) \, ds  + \sqrt{\frac{2}{N}} \int_{-N^\alpha}^{w} \pa_{\omega} E_N^2 \Big( x_0 + \sqrt{\tfrac{2}{N}}z, x_0 + \sqrt{\tfrac{2}{N}}s \Big) \, ds.
\end{split}
\end{align}

We first show that for $z=\OO(N^\nu)$ uniformly,
    \begin{equation} \label{EN2 asymp edge 1}
    E_N^2\Big( x_0+\sqrt{\tfrac{2}{N}}z, 0 \Big) = \sqrt{2} \, e^{-2z^2} \Big[ 1 + \sqrt{\tfrac{2}{N}}\curvature_2(\tfrac{2}{3}z^3-\tfrac{\tau}{1+\tau}z) + \OO(N^{-1+6\nu}) \Big].
    \end{equation}
Note that by \eqref{Hermite strong asymptotic expansion} and Stirling's formula, we have
\begin{equation*}
\begin{split}
\frac{(\tau/2)^N}{(2N-1)!!} H_{2N}\Big( \sqrt{N\tfrac{1-\tau^2}{2\tau}} \xi \Big) &= \frac{1}{(2\pi (1-\tau^2) N)^{1/4}} \frac{\sqrt{(2N)!}}{(2N-1)!!} \sqrt{\psi'(\xi)} e^{N(2g(\xi)-l/2)} \Big(1+\OO(1/N) \Big) 
\\
& = \frac{1}{(2(1-\tau^2))^{1/4}} \sqrt{\psi'(\xi)} e^{N(2g(\xi)-l/2)} \Big(1+\OO(1/N)\Big), 
\end{split}
\end{equation*}
where the $\OO(1/N)$-term is uniform for $\xi$ in a compact subset of $\C \setminus [-F_0, F_0]$.
Applying the estimates \eqref{Lemma A1 2} and \eqref{Lemma A1 3}, the desired asymptotic behaviour \eqref{EN2 asymp edge 1} follows from the expression \eqref{E_N^2} and the summation formula \eqref{E_N^2 Taylor sum}. 

Next, we show that for $z=\OO(N^\nu)$ uniformly, 
    \begin{equation} \label{EN2 asymp edge 2}
    \int_{0}^{x_0-\sqrt{\frac{2}{N}}N^\alpha} \pa_\omega E_N^2\Big( x_0 + \sqrt{\tfrac{2}{N}}z, s \Big) \, ds = e^{-2z^2} \OO(N^{ \frac23 }e^{-N^{2\alpha}}).
    \end{equation}
For this, we derive the asymptotic behaviour of $\pa_\omega E_N^2(\xi, s)$ using Lemma~\ref{Lem_ EN1 EN2 deri}. 
Note that by \eqref{Hermite strong asymptotic bound} and \eqref{Hermite strong asymptotic expansion}, we have 
\begin{align}
\abs{\pa_\omega E_N^2(\xi, s)} &= 2(1-\tau) (\tfrac{2}{\pi(1-\tau^2)})^{\frac14} \abs{\psi'(\xi)} 
\\
&\quad \times \exp\Big( N \re \Big[-\tfrac{1-\tau}{2}\xi^2 +2g(\xi)-\tfrac{l}{2}-\tfrac{1-\tau}{2}s^2 +2g(s) -\tfrac{l}{2} \Big] \Big)\, \OO(N^{\frac23}).
\end{align} 
Moreover, it follows from \eqref{Lemma A1 1}, \eqref{Lemma A1 2} and \eqref{Lemma A1 3} that
\begin{equation}
\Big| \pa_{\omega} E_N^2 \Big( x_0+\sqrt{\tfrac{2}{N}}z, s \Big) \Big| = e^{-2\re z^2} \OO(N^{ \frac23 } e^{-N^{2\alpha}})
\end{equation}
uniformly for $0\leq s\leq x_0-\sqrt{\frac{2}{N}}N^{\alpha}$ and $z=\OO(N^\nu)$.
Now by using the triangle inequality and integrating this asymptotic, we obtain \eqref{EN2 asymp edge 2}. 

Finally, we show that for $z, w = \OO(N^\nu)$ uniformly, 
    \begin{align}
    \begin{split}  \label{EN2 asymp edge 3}
    &\quad \sqrt{\tfrac{2}{N}} \int_{-N^\alpha}^{w} \pa_\omega E_N^2\Big( x_0 + \sqrt{\tfrac{2}{N}}z, x_0 + \sqrt{\tfrac{2}{N}}s \Big) \, ds 
    = -\tfrac{1}{\sqrt{2}} e^{-2z^2} (1+\erf(\sqrt{2}w))
    \\
    &\quad -\tfrac{\curvature_2}{\sqrt{N}} e^{-2z^2} \Big[ (\tfrac{2}{3}z^3-\tfrac{\tau}{1+\tau}z)(1+\erf(\sqrt{2}w)) +\tfrac{1}{\sqrt{2\pi}}(\tfrac{1}{1+\tau}-\tfrac{1+2w^2}{3})e^{-2w^2} \Big]  
    \\
    &\quad + e^{-2z^2 + 2(\im w)^2} \OO(N^{-1+7\nu}).
    \end{split}
    \end{align}
Similarly as above, we use Lemma~\ref{Lem_ EN1 EN2 deri} and the asymptotics \eqref{Hermite strong asymptotic expansion},  \eqref{Lemma A1 2}, which leads to
$$
\pa_{\omega} E_N^2\Big( x_0+\sqrt{\tfrac{2}{N}}z, x_0+\sqrt{\tfrac{2}{N}}s \Big) = -\sqrt{\tfrac{2N}{\pi}} e^{-2(z^2+s^2)} \Big[ 1 + \sqrt{\tfrac{2}{N}} \, \curvature_2 \big( \tfrac{2}{3}(z^3+s^3) - \tfrac{s+\tau z}{1+\tau} \big) + \OO(N^{-1+6\nu}) \Big].
$$
Here $\curvature_2$ is given by \eqref{KT LR} with $T=2$. By integration, \eqref{EN2 asymp edge 3} follows from elementary asymptotic behaviour of the error function (see e.g.\ \cite[Eq.(7.12.1)]{olver2010nist}).  

Combining \eqref{EN2 edge decom} with the asymptotics \eqref{EN2 asymp edge 1}, \eqref{EN2 asymp edge 2} and \eqref{EN2 asymp edge 3} and using that $z, w$ are bounded (but $\nu > 0$), we obtain the desired expansion \eqref{EN2 asym edge} for $E_N^2$. Choosing $\eps > 7 \nu$ completes the proof.
\end{proof}

\subsection{Outside the bulk} 

This subsection is devoted to proving the last statement of Proposition~\ref{Prop_rN}. Again, this is an immediate consequence of the following lemma. 

\begin{lem} \label{Error terms outside real bulk}
Let $x_0 \in \R$ with $\abs{x_0} > \sqrt{2 \frac{1+\tau}{1-\tau}}$. Then there exists a neighbourhood $U \subset \C$ of $x_0$ and an $\eps > 0$ such that the following estimates hold uniformly for $\xi, \omega \in U$:
\begin{equation*}
E_N^1(\xi, \omega) = \OO(e^{-N\eps}), \qquad
E_N^2(\xi, \omega) = \OO(e^{-N\eps}).
\end{equation*}
\end{lem}

In the proof of Lemma~\ref{Error terms outside real bulk} we shall use an estimate for Hermite polynomials.

\begin{lem} \label{Lem_Hermite ineq}
For all $\tau \in (0, 1)$, $0 \leq l \leq N-1$ and $\abs{\omega_0} \geq \sqrt{2\frac{1+\tau}{1-\tau}}$, we have
\begin{equation} \label{Hermite ineq}
\Big| \frac{(\tau/2)^l}{(2l)!!} H_{2l}\Big( \sqrt{N\tfrac{1-\tau^2}{2\tau}} \omega_0 \Big) \Big|
\leq \Big| \frac{(\tau/2)^{l+1}}{(2l+2)!!} H_{2l+2}\Big( \sqrt{N\tfrac{1-\tau^2}{2\tau}} \omega_0 \Big) \Big|.
\end{equation}
\end{lem}

The proof of this lemma is given in Appendix~\ref{Appendix 2}. 

\begin{proof}[Proof of Lemma~\ref{Error terms outside real bulk}]
The proof is merely same as that of Lemma~\ref{Error terms in real bulk} except the fact that in the representation \eqref{eq:E_N^12 integral decomposition}, we choose $\xi_0$, $\omega_0$  outside of the droplet.
In particular, the estimates of the terms involving the derivatives of $E_N^1$ or $E_N^2$ are the same as those in Lemma~\ref{Error terms in real bulk}. Thus all we need to analyse are the initial values.

Let us choose a compact subset $V$ of $K^\mathsf{c}_2$.
For $E_N^1$, it again follows from \cite[Lemma D.4]{MR3450566} that for an $\eps>0,$
\begin{equation*}
E_N^1(\bar{\omega}, \omega)
= \OO(e^{-N\eps}), \qquad (\omega \in V).
\end{equation*}

It remains to analyse $E_N^2$ in \eqref{E_N^2}. 
For the $\xi$-dependent part of $E_N^2$, the estimates above work, but the $\omega$-dependent part of $E_N^2$ requires special care.
Let us choose $\omega_0 = x_0$. 
Then by combining Lemma~\ref{Lem_Hermite ineq} and the asymptotics \eqref{Hermite strong asymptotic bound}, we have
\begin{align}
\begin{split}
\Big| e^{ -\tfrac{1}{2}N(1-\tau) x_0^2 } \sum_{l = 0}^{N - 1} \frac{(\tau/2)^l}{(2l)!!} H_{2l}\Big( \sqrt{N\tfrac{1-\tau^2}{2\tau}} x_0 \Big) \Big|
&\leq e^{ -\tfrac{1}{2}N(1-\tau) x_0^2 } N \frac{(\tau/2)^N}{(2N)!!}  \Big| H_{2N}\Big( \sqrt{N\tfrac{1-\tau^2}{2\tau}} x_0 \Big) \Big|
\\
&= e^{-N\re [ \frac{1}{2}(1-\tau)x_0^2 -2g(x_0) +\frac{l}{2} ] } \sqrt{\tfrac{(2N-1)!!}{(2N)!! \, N}} \, \OO(N^{1+ \frac{5}{12} }).
\end{split}
\end{align}
Since $x_0$ is real, we have
\begin{equation*}
\Re\Big[ \tfrac{1}{2}(1-\tau)x_0^2 -2g(x_0) +\tfrac{l}{2} \Big]
= \frac{1}{2} \Omega(x_0)
> \eps.
\end{equation*}
Then by Stirling's formula, we obtain
\begin{equation*}
\exp\Big( -\tfrac{1}{2}N(1-\tau) x_0^2 \Big) \sum_{l = 0}^{N - 1} \frac{(\tau/2)^l}{(2l)!!} H_{2l}\Big( \sqrt{N\tfrac{1-\tau^2}{2\tau}} x_0 \Big)
= \OO(N^{ \frac23 } e^{-N\eps}), 
\end{equation*}
which leads to
\begin{equation*}
E_N^2(\xi, x_0) = 2\sqrt{1-\tau^2} \, \OO(N^{ \frac16 } e^{-\frac{1}{2}N\eps}) \OO(N^{ \frac23 }e^{-N\eps}) = \OO(N^{ \frac56 }e^{-\frac{3}{2}N\eps}).
\end{equation*}
This completes the proof.
\end{proof}

\subsection{Proof of the main theorem} \label{Subsec_main thm proof}

In this subsection, we complete the proof of Theorem~\ref{Thm_local ell}. 
By Corollary~\ref{Cor_ODE transform} and Proposition~\ref{Prop_rN}, it remains to solve the resulting differential equations. For this, we first need the following lemma, which readily follows from the basic theory of ordinary differential equations. 

\begin{lem} \label{Lem_ODE rate}
Let $a, r_N: \C \to \C$ be analytic functions, $f_0 \in \C$ and for all $N \in \N$ let $f_N: \C \to \C$ be the solution to
\begin{equation*}
f_N' = a\, f_N + r_N, \qquad f_N(z_0) = f_0.
\end{equation*}
If $r_N \to r$ uniformly on compact subsets of $\C$, then the limit $f(z) := \lim_{N \to \infty} f_N(z)$ exists for all $z \in \C$, is analytic and satisfies
\begin{equation*}
f' = a\, f + r, \qquad f(z_0) = f_0.
\end{equation*}
Moreover suppose that as $N\to \infty$, the inhomogeneous term $r_N$ has the asymptotic expansion 
\begin{equation*}
r_N = r + r^{(\gamma)}/N^\gamma + \OO(1/N^{\wt{\gamma}}), \qquad (\wt{\gamma}>\gamma), 
\end{equation*}
where $\OO(1/N^{\wt{\gamma}})$-term is uniform on compact subsets of $\C$.
Then
\begin{equation*}
f_N = f + f^{(\gamma)}/N^\gamma + \OO(1/N^{\wt{\gamma}})
\end{equation*}
also holds uniformly on compact subsets of $\C$ and the subleading term $f^{(\gamma)}$ satisfies
\begin{equation*}
f^{(\gamma)\prime} = a\, f^{(\gamma)} + r^{(\gamma)}, \qquad f^{(\gamma)}(z_0) = 0.
\end{equation*}
\end{lem}

\begin{proof}
Recall that the solution of our linear inhomogeneous differential equation is given by the variation-of-constants formula
\begin{equation*}
    f_N(z) = U(z, z_0) f_0 + \int_{z_0}^{z} U(z, t) r_N(t) \, dt, \qquad 
    U(z, t) := \exp\Big( \int_{t}^{z} a(s) \, ds \Big).
\end{equation*}
For any compact subset $K$ of $\C$, there exists $\rho_K > 0$ such that
\begin{equation*}
    \abs{r(z) + r^{(\gamma)}(z) / N^\gamma - r_N(z)} \leq \rho_K/N^{\wt{\gamma}}.
\end{equation*}
for all $N \in \N.$
Set $f(z) = U(z, z_0) f_0 + \int_{z_0}^{z} U(z, t) r(t) \, dt$ and $f^{(\gamma)}(z) = \int_{z_0}^{z} U(z, t) r^{(\gamma)}(t) \, dt$. 
Then one can show that the following inequality holds for all $z \in K$:
\begin{equation*}
\begin{split}
    \abs{f(z) + f^{(\gamma)}(z) / N^\gamma - f_N(z)}
    & \leq \abs{z - z_0} \cdot \max_{t \in K} \abs{U(z, t)} \cdot \max_{t \in K} \abs{r(t) + r^{(\gamma)}(t) / N^\gamma - r_N(t)} \\
    & \leq \diam(K) \, e^{A \cdot \diam(K)} \rho_K/N^{\wt{\gamma}},
\end{split}
\end{equation*}
where $A := \max_{s \in K} \re a(s) < \infty$.
Hence $f_N$ has the same rate of convergence as $r_N$.
\end{proof}

Using Lemma~\ref{Lem_ODE rate} with $\gamma = 1/2$ and $\wt{\gamma}=1$, we can now derive the large-$N$ asymptotics of $\widehat{\kappa}_N(z,w)$. 
For the proof of Theorem~\ref{Thm_local ell}, we recall a notion of the equivalent kernels. (See also \cite{akemann2021scaling}.)

\begin{rmk*} (Equivalent kernels)
Consider a $2k\times2k$ skew-symmetric matrix $A = (a_{j,l})_{j,l}$ and let $g_j \in \C$ for $1 \leq j \leq 2k$.
The new matrix $B = (g_j g_l a_{j,l})_{j,l}$ is again skew-symmetric and we have $\Pf(B) = \prod_{j=1}^{2k} g_j \, \Pf(A)$.
Furthermore, if $g_{2l} = 1/g_{2l-1}$ for $1 \leq l \leq k$, then $\Pf(B) = \Pf(A)$.
This implies that different kernels can give rise to the same correlation functions.
In particular, we call two pre-kernels $\bfkappa$ and $\widetilde{\bfkappa}$ \emph{equivalent} if there exists a unimodular function $g:\C \to \C$ with $g(\overline{\zeta}) = 1/g(\zeta)$ such that $\widetilde{\bfkappa}(\zeta, \eta) = g(\zeta) g(\eta) \bfkappa(\zeta, \eta)$.
Also, we call the combination $c(\zeta, \eta) := g(\zeta) g(\eta)$ a \emph{cocycle}.
\end{rmk*}

\begin{proof}[Proof of Theorem~\ref{Thm_local ell}]
Recall that by definition \eqref{transformed kernel}, we have $\kappa_N= \widehat{\kappa}_N/ \omega_N$.
Note that 
$$
e^{-\frac{N}{2} ( Q( p+\frac{z}{\sqrt{N \delta}} )+Q( p+ \frac{w}{\sqrt{N \delta}} ) )} \frac{1}{\omega_N(z,w)}
=e^{-\abs{z}^2 -\abs{w}^2 +2zw}  \, \frac{1}{c_N(z,w)},
$$
where the cocycle $c_N(z,w)$ is given by 
\begin{equation}
c_N(z,w)=\exp\Big( -i\sqrt{2N\tfrac{1-\tau}{1+\tau}}\, p \, \im z - i\sqrt{2N\tfrac{1-\tau}{1+\tau}}\, p \, \im w +i\tau\im z^2 +i\tau \im w^2 \Big).
\end{equation}
Therefore as $N \to \infty$, the equivalent kernels have the uniform limit
\begin{equation} \label{up to cocycle}
\lim_{N\to\infty} c_N(z,w) e^{-\frac{N}{2} ( Q( p+\frac{z}{\sqrt{N \delta}} )+Q( p+ \frac{w}{\sqrt{N \delta}} ) )} \kappa_N(z,w)
= e^{-\abs{z}^2 -\abs{w}^2 +2zw} \, \widehat{\kappa}(z,w).
\end{equation}

By Proposition~\ref{Prop_rN}, one can see that for the leading order, the differential equation \eqref{ODE tau transformed with error terms} coincides with the one in \cite{MR1928853} (for the real bulk case) and in \cite{akemann2021scaling} (for the real edge case). Moreover, it was shown that the kernels \eqref{kappa bulk} and \eqref{kappa edge} are the associated solution in each case. (Here the uniqueness follows from the anti-symmetry of the pre-kernel.)
On the other hand, for the case outside the droplet the differential equation is trivial. 
Therefore only the derivation of $\kappa_{\textup{edge}}^{\R,1/2}$ in \eqref{kappa edge sub} remains.

By Corollary~\ref{Cor_ODE transform}, Proposition~\ref{Prop_rN}, Lemma~\ref{Lem_ODE rate} and \eqref{up to cocycle}, the function 
$$\widehat{\kappa}_{\textup{edge}}^{\R,1/2}(z,w):=e^{-2zw}\,\kappa_{\textup{edge}}^{\R,1/2}(z,w)
$$ 
satisfies the differential equation
\begin{equation}
\pa_z \widehat{\kappa}_{\textup{edge}}^{\R,1/2}(z,w)=2(z-w) \widehat{\kappa}_{\textup{edge}}^{\R,1/2}(z,w)+r^{(1/2)}(z,w). 
\end{equation}
This differential equation can be easily solved by virtue of an integrating factor. 
Furthermore, the anti-symmetry $ \widehat{\kappa}_{\textup{edge}}^{\R,1/2}(z,w)=-\widehat{\kappa}_{\textup{edge}}^{\R,1/2}(w,z)$ determines the solution as 
\begin{equation*}
\begin{split}
\widehat{\kappa}_{\textup{edge}}^{\R,1/2}(z,w) &= \tfrac{1}{\sqrt{2}}(\tfrac{1+\tau}{1-\tau})^{\frac32} e^{(z-w)^2}  
\\
&\quad \times \int_{w}^{z} e^{-2t^2} \Big[ \tfrac{1}{\sqrt{2\pi}} (\tfrac{4}{3}t^2-\tfrac{4}{3}tw+\tfrac{2}{3}w^2-\tfrac{\tau}{1+\tau})e^{-2w^2} - (\tfrac{2}{3}t^3-\tfrac{\tau}{1+\tau}t)\erfc(\sqrt{2}w) \Big]\, dt \\
&=\tfrac{1}{12\sqrt{2}}(\tfrac{1+\tau}{1-\tau})^{\frac32} e^{(z-w)^2} \Big[ \Big((2z^2+\tfrac{1-2\tau}{1+\tau})e^{-2z^2}\erfc(\sqrt{2}w) + 2\sqrt{\tfrac{2}{\pi}}\,w e^{-2(z^2+w^2)}\Big) - \Big(z \leftrightarrow w \Big) \Big],
\end{split}
\end{equation*}
which completes the proof. 
\end{proof}

\begin{appendix}

\section{Some preliminary estimates} \label{Appendix}

In this appendix, we show some preliminary estimates used in Section~\ref{sec:Strong Limit}. 

\subsection{Proof of Lemma~\ref{Estimates for E_N^2 at edge}} \label{Appendix 1}

The proof of Lemma~\ref{Estimates for E_N^2 at edge} merely follows from the computations given in \cite{MR3450566}. 

First, we prove \eqref{Lemma A1 1}. 
By definitions \eqref{g function LR}, \eqref{Omega LR}, we have that for $s \in \R$,
\begin{equation*}
\re\Big[ \tfrac{1-\tau}{2}s^2 -Tg(s) +\tfrac{l}{2} \Big] = \frac{1}{2} \Omega(s).
\end{equation*}
Then \cite[Lemma~3.3]{MR3450566} implies the existence of a threshold $\delta > 0$ such that the approximation $$\Omega(s) = 2\abs{s-x_0}^2 + \OO(\abs{s-x_0}^3)$$ 
holds uniformly for $s \in [x_0-\delta, x_0+\delta]$.
Additionally,
recall that $\Omega(s) \geq 0$ for all $s$ and $\Omega(s)=0$ only for $s$ on the boundary of the droplet, i.e.\ $s=x_0$.
Therefore we have an $\eps>0$ with $\frac{1}{2}\Omega(s)\geq\eps$ for all $s\in[0, x_0-\delta]$.
For $s\in[x_0-\delta, x_0-\sqrt{2/N}N^\alpha]$ we can use the quadratic approximation from above to derive
\begin{equation}
\frac{1}{2}\Omega(s) \geq 2N^{-1+2\alpha} + \OO(N^{-3/2+3\alpha}) \geq N^{-1+2\alpha}.
\end{equation}
Here the last inequality holds for large enough $N$ if $\alpha<1/2$.

Next, we show the second assertion \eqref{Lemma A1 2}.
By \cite[Lemma~3.2]{MR3450566}, we have the expansion 
\begin{equation*}
Tg \Big( x_0+\sqrt{\tfrac{2}{N}}z \Big) = \frac{1-\tau}{2} x_0^2 + \frac{l}{2} + (1-\tau) x_0 \sqrt{\frac{2}{N}} z - \frac{1+\tau}{N} z^2 + \frac{2\sqrt{2}\curvature_T}{3} \frac{z^3}{N^{3/2}} + \OO\Big(\frac{z^4}{N^2}\Big).
\end{equation*}
Thus we obtain
\begin{equation*}
\exp\Big( -N\Big[ -\tfrac{1-\tau}{2} \Big( x_0+\sqrt{\tfrac{2}{N}}z \Big)^2 + Tg \Big( x_0+\sqrt{\tfrac{2}{N}}z \Big) -\frac{l}{2} \Big] \Big)
= e^{-2z^2} \exp\Big[ \tfrac{2\sqrt{2}}{3}\curvature_T \tfrac{z^3}{\sqrt{N}} + \OO\Big( \tfrac{z^4}{N} \Big)\Big],
\end{equation*}
which leads to \eqref{Lemma A1 2}.

For the last assertion \eqref{Lemma A1 3}, we use Taylor expansion around $x_0$ analogous to \cite[Lemma~3.5]{MR3450566}. 
Note that for $x_0 \in \R$, we have
\begin{equation}
\psi(x_0) = 1, \qquad
\psi'(x_0) = \sqrt{\frac{1+\tau}{T(1-\tau)}}, \qquad
\psi''(x_0) = -\frac{2\tau(1+\tau)}{T(1-\tau)^2}.
\end{equation}
By direct computations, this gives rise to \eqref{Lemma A1 3} using the convention $\curvature_T= \frac{1}{\sqrt{T}}(\frac{1+\tau}{1-\tau})^{3/2}$.

\subsection{Proof of Lemma~\ref{Lem_Hermite ineq}} \label{Appendix 2} 

It is well known that the zeros of the Hermite polynomial of degree $n$ lie in the interval $(-\sqrt{2n+1}, \sqrt{2n+1})$, see e.g.\ \cite[Section 18.16]{olver2010nist}. 
Thus we have excluded all zeros.
Additionally since the Hermite polynomials are positive as $\omega_0 \to \pm \infty$, we can omit the absolute values.

We simplify the inequality \eqref{Hermite ineq} by multiplying both sides with $(2l)!! \, (2/\tau)^l$ and applying the three term recurrence relation \eqref{three term} to $H_{2l+2}$.
Let us rewrite the argument of the polynomials as 
$$
\sqrt{N(1-\tau^2)/(2\tau)} \, \omega_0=(1+\tau)y/\sqrt{\tau}.
$$
Then Lemma~\ref{Lem_Hermite ineq} follows from the fact that for $\tau \in (0, 1)$, $0 \leq l \leq N-1$ and $y \geq \sqrt{N}$,
\begin{equation*}
\Big( 1 + \frac{2l+1}{2l+2}\tau \Big) H_{2l}\Big( \tfrac{1+\tau}{\sqrt{\tau}}y \Big) \leq \frac{\tau (1+\tau)}{2(l+1)}y \, H_{2l+1}\Big( \tfrac{1+\tau}{\sqrt{\tau}}y \Big).
\end{equation*}

Let us show the slightly stricter inequality
\begin{equation}\label{Hermite ineq v2}
H_{2l}\Big( \tfrac{1+\tau}{\sqrt{\tau}}y \Big) \leq \frac{\tau}{2(l+1)}y \, H_{2l+1}\Big( \tfrac{1+\tau}{\sqrt{\tau}}y \Big)
\end{equation}
by an induction argument.  
For $l = 0$, the inequality \eqref{Hermite ineq v2} is obvious since $H_0(x) = 1$ and $H_1(x) = 2x$. 

From now on, we shall write $a_l=\tau y/(2l+2)$ and $x=(1+\tau)y/\sqrt{\tau}$ to lighten the notations.  
Due to the three term recurrence relation \eqref{three term}, it suffices to show that 
\begin{equation}
H_{2l}(x) \leq a_l H_{2l+1}(x) = a_l \Big[2xH_{2l}(x)-4lH_{2l-1}(x)\Big]. 
\end{equation}
Again, by the three term recurrence relation, this is equivalent to 
\begin{equation}
H_{2l-2}(x) \leq \frac{2a_l(x^2-l)-x}{(2l-1)(2a_lx-1)} H_{2l-1}(x).
\end{equation}
By the induction hypothesis, we have  $H_{2l-2}(x) \leq a_{l-1} H_{2l-1}(x)$. 
Therefore to complete the induction step it is enough to show that
\begin{equation} \label{aux estimate}
a_{l-1} \leq \frac{2a_l(x^2-l)-x}{(2l-1)(2a_lx-1)} , \qquad \text{i.e.} \qquad  \frac{\sqrt{\tau}}{2l} \leq \frac{(1+\tau)^2y^2-\tau l-(1+\tau)(l+1)}{\sqrt{\tau}(2l-1)\big((1+\tau)y^2-(l+1)\big)}.
\end{equation}
This inequality \eqref{aux estimate} follows from the straightforward computations using elementary algebra and we leave it to the interested readers.

\subsection*{Acknowledgements} The authors are greatly indebted to Gernot Akemann for suggesting the problem, careful reading, and much appreciated help with improving this manuscript.
It is also our pleasure to thank Seung-Yeop Lee and Roman Riser for helpful discussions. 
The present work was completed when the authors visited the Department of Mathematical Sciences at Korea Advanced Institute of Science and Technology, and we wish to express our gratitude to Ji Oon Lee for the invitation and hospitality.

\end{appendix}


\bibliographystyle{abbrv}
\bibliography{RMTbib}
\end{document}